\documentclass[a4paper,11pt,draft]{article}
\usepackage{amsmath,amssymb,enumerate,color}
\usepackage{amsthm,color}
\usepackage{txfonts}
\usepackage{bigints}
\usepackage{comment}
\usepackage{tabularx}
\usepackage{longtable}

\newcommand{\R}{\mathbb{R}}
\newcommand{\C}{\mathbb{C}}
\newcommand{\Z}{\mathbb{Z}}
\newcommand{\N}{\mathbb{N}}

\newtheorem{theorem}{Theorem}[section]
\newtheorem{lemma}[theorem]{Lemma}
\newtheorem{proposition}[theorem]{Proposition}
\newtheorem{corollary}[theorem]{Corollary}
\theoremstyle{remark}
\newtheorem{remark}{Remark}[section]
\theoremstyle{definition}

\newtheorem{definition}{Definition}[section]

\numberwithin{equation}{section}

\makeatletter
\def\@cite#1#2{[{{\bfseries #1}\if@tempswa , #2\fi}]}
\makeatother                  
\begin{document}
\begin{center}
\Large{{\bf
Variational problems for the system of nonlinear Schr\"odinger equations with derivative nonlinearities}}
\end{center}

\vspace{5pt}

\begin{center}
Hiroyuki Hirayama%
\footnote{
Faculty of Education, University of Miyazaki, 1-1, Gakuenkibanadai-nishi, Miyazaki, 889-2192 Japan,  
E-mail:\ {\tt h.hirayama@cc.miyazaki-u.ac.jp}}
and 
Masahiro Ikeda%
\footnote{
Department of Mathematics, Faculty of Science and Technology, Keio University, 3-14-1 Hiyoshi, Kohoku-ku, Yokohama, 223-8522, Japan/Center for Advanced Intelligence Project, RIKEN, Japan, 
E-mail:\ {\tt masahiro.ikeda@keio.jp/masahiro.ikeda@riken.jp}}
\end{center}

\newenvironment{summary}{\vspace{.5\baselineskip}\begin{list}{}{%
     \setlength{\baselineskip}{0.85\baselineskip}
     \setlength{\topsep}{0pt}
     \setlength{\leftmargin}{12mm}
     \setlength{\rightmargin}{12mm}
     \setlength{\listparindent}{0mm}
     \setlength{\itemindent}{\listparindent}
     \setlength{\parsep}{0pt}
     \item\relax}}{\end{list}\vspace{.5\baselineskip}}
\begin{summary}
{\footnotesize {\bf Abstract.}
We consider the Cauchy problem of the system of 
nonlinear Schr\"odinger equations with derivative nonlinearlity. This system was introduced by Colin-Colin (2004) 
as a model of laser-plasma interactions. 
We study existence of ground state 
solutions and the global well-posedness 
of this system by using the variational methods. 
We also consider the stability of traveling waves for this system. 
These problems are proposed by Colin-Colin as the open problems. 
We give a subset of the ground-states set 
which satisfies the condition of stability. 
In particular, we prove the stability 
of the set of traveling waves with small speed 
for $1$-dimension. 
}
\end{summary}

{\footnotesize{\it Mathematics Subject Classification}\/ (2020): %
35Q55; 
35A01; 
35A15;
35B35
}\\
{\footnotesize{\it Key words and phrases}\/: %
System of Schr\"odinger equations,
Derivative nonlinearity,
Global well-posedness,
Ground state,
Traveling waves, 
Stability
}

\section{Introduction}
\ \ \ In the present paper we study well-posedness, asymptotic behavior of solutions and stability of traveling waves to the Cauchy problem of 
the system of the Schr\"odinger equations with derivative nonlinearities on the Euclidean space $\mathbb{R}^d$:
\begin{equation}\label{NLS}
\begin{cases}
i\partial_tu_1+\alpha \Delta u_1=-(\nabla \cdot u_3)u_2,& t\in I_{\max},\ x\in \R^d,\\
i\partial_tu_2+\beta \Delta u_2=-(\nabla \cdot \overline{u_3})u_1,& t\in I_{\max},\ x\in \R^d,\\
i\partial_tu_3+\gamma \Delta u_3=\nabla (u_1\cdot \overline{u_2}),& t\in I_{\max},\ x\in \R^d,\\
(u_1,u_2,u_3)|_{t=0}=(u_{1,0},u_{2,0},u_{3,0}),& x\in \R^d,
\end{cases}    
\end{equation}
where $d\in \{1,2,3\}$ is the spatial dimension, $\partial_t:=\partial/\partial_t$ is the time derivative, $\partial_{x_j}:=\partial/\partial_{x_j}$ is the spatial derivative with respect to the spatial variable $x_j$ with $j\in \{1,\cdots,d\}$, $\nabla:=(\partial_{x_1},\cdots,\partial_{x_d})$ is the nabla, and $\Delta:=\sum_{j=1}^d\partial_{x_j}^2$ is the Laplace operator on $\R^d$. Here $\alpha$, $\beta$, $\gamma\in \R\backslash \{0\}$ are real constants, $I_{\max}:=(-T_{\min},T_{\max})$ is the maximal existence time interval, $T_{\max}\in (0,\infty]$ (resp. $-T_{\min}$) is the forward (resp.backward) existence time of the function $(u_1,u_2,u_3)$, $u_1$, $u_2$, $u_3:I_{\max}\times\R^d\rightarrow \C^d$ are unknown $d$-dimensional complex vector-valued functions and $u_{1,0}$, $u_{2,0}$, $u_{3,0}:\R^d\rightarrow \C^d$ are prescribed $d$-dimensional complex vector-valued functions. 
Let $u_j^{(k)}:I_{\max}\times\mathbb{R}^d\rightarrow \C$ denote the $k$-th component 
of $u_j$ for $j=1,2,3$ and $k=1,\cdots,d$. 
Namely, $u_j=(u_j^{(1)},\cdots, u_j^{(d)})$. 

The system (\ref{NLS}) was first derived from the bi-fluid Euler-Maxwell system by Colin and Colin \cite[P301]{CC04} as a model of laser-plasma interaction, in which $u_1$ and $u_3$ denote the scattered light and the electronic-plasma wave, respectively and $u_2$ is the sum of the incident laser field and the gauged Brillouin component.

The system (\ref{NLS}) is invariant under the following scale transformation
\[
U_{\lambda}(t,x):=\lambda^{-1}U(\lambda^{-2}t,\lambda^{-1}x),\ \ \ (U=(u_1,u_2,u_3),\ \lambda>0).
\]
More presicely if $U$ is a solution to (\ref{NLS}), then so is $U_{\lambda}$ with the rescaled initial data $\lambda^{-1}U_0(\lambda^{-1}x)$. Moreover, we calculate
\[
\|U_{\lambda}(0)\|_{\dot{H}^s}=\lambda^{\frac{d}{2}-1-s}\|U_0\|_{\dot{H}^s},
\]
where $\|\cdot\|_{\dot{H}^s}$ is the $s$-th order and $L^2$-based homogeneous Sobolev norm.
If $s$ satisfies $s=d/2-1$, then $\|U_{\lambda}(0)\|_{\dot{H}^s}=\|U_0\|_{\dot{H}^s}$ for any $\lambda>0$. Therefore $d/2-1$ is called the scaling critical Sobolev index and the current cases where $d=1,2,3$ belong to the $H^1$(energy)-subcritical case.

For $s\in \R$, we denote the inhomogeneous $L^2$-based $s$-th order Sobolev space by $H^{s}=H^{s}(\mathbb{R}^d)$ endowed with the norm
\[
\Vert f\Vert _{H^{s}}:=\left\|\langle\nabla\rangle^sf\right\|_{L^2}=\left\Vert \langle \xi \rangle ^{s}\hat{f}\right\Vert _{L^{2}},
\]
where $\langle \cdot \rangle :=1+\vert \cdot \vert$. 
We introduce the function space $\mathcal{H}^s:=(H^s(\R^d))^{3d}$ 
and 
give the initial data $(u_{1,0},u_{2,0},u_{3,0})$ belonging to $\mathcal{H}^s$. 
In \cite{CC04, HKO22}, the well-posedness of (\ref{NLS}) in $\mathcal{H}^s$ 
was considered by using the energy method, 
and they obtained the local well-posedness for $s>d/2+3$ under the condition $\beta \gamma > 0$. 
The global well-posedness for sufficiently small initial data 
was also obtained in \cite{HKO22} if the condition $\beta +\gamma \ne 0$ holds. 
On the other hand, the low regularity well-posedness of (\ref{NLS}) in $\mathcal{H}^s$ 
was studied in \cite{Hirayama14, HK19, HHO20, HKO21} by using the iteration argument 
in the space of the Fourier restriction norm. 
In particular, the well-posedness in the energy space $\mathcal{H}^1$ was obtained 
as follows. 
\begin{theorem}[Local well-posedness in $\mathcal{H}^1$, \cite{Hirayama14, HKO21}]\label{NLS_LWP_Known}
Let $d=1,2$, or $3$. 
Assume that $\alpha,\beta,\gamma \in \R\backslash \{0\}$ 
satisfy $(\alpha -\gamma)(\beta +\gamma)\ne 0$. 
Then, {\rm (\ref{NLS})} is locally well-posed in $\mathcal{H}^1$. The forward (resp. backward) maximal existence time $T_{\max}$ (resp. $T_{\min}$) depends only on $\|U_0\|_{\mathcal{H}^1}$. Furthermore, if $\alpha$, $\beta$, and $\gamma$ have same sign, 
then the local solution for sufficiently small initial data 
in $\mathcal{H}^1$
can be extended globally in time. 
\end{theorem}
\begin{remark}
The global well-posedness in Theorem~\ref{NLS_LWP_Known} 
was obtained by using the conservation law 
of the charge and the energy, which will be defined below. 
\end{remark}
\begin{remark}\label{gwp_known_rem}
    For the case $d=1$, which is $L^2$-subcritical case, 
    we can obtain the global well-posedness in $\mathcal{H}^1$ 
    without smallness condition on initial data 
    by using the Gagliardo-Nirenberg-Sobolev inequality. 
    However for the case $d=2$ or $3$, 
    the global well-posedness without smallness condition 
    for initial data was not known. 
\end{remark}
\begin{remark}\label{gwp_known_rem_2}
    For the case $d=4$, which is $H^1$-critical case, 
    the global well-posedness and the scattering of solution 
    for small initial data in $\mathcal{H}^1$
    are obtained in \cite{Hirayama14}
\end{remark}

As mentioned above remarks, 
there are only a few results for global solutions.  
In \cite{CC04}, the authors proposed open problems, which are existence and stability of solitary waves for the system (\ref{NLS}). We address these problems in the case of $d=1,2,3$ in this paper.

For vector valued functions $f=(f^{(1)},\cdots,f^{(d)})$ and $g=(g^{(1)},\cdots,g^{(d)})\in (H^1(\R^d))^d$, 
we define the derivatives, gradient, $L^2$-norm, $H^1$-norm, and $L^2$-inner product as
\[
\begin{split}
&\partial_lf:=(\partial_lf^{(1)},\cdots,\partial_lf^{(d)})\in (L^2(\R^d))^d,\ \ 
\nabla f:=(\partial_1f,\cdots ,\partial_df)\in (L^2(\R^d))^{d\times d},\\
&\|f\|_{L^2(\R^d)}^2:=\sum_{k=1}^d\|f^{(k)}\|_{L^2(\R^d)}^2,\ \ 
\|\nabla f\|_{L^2 (\R^d)}^2:=\sum_{l=1}^d\|\partial_l f\|_{L^2(\R^d)}^2
=\sum_{l=1}^d\sum_{k=1}^d\|\partial_lf^{(k)}\|_{L^2(\R^d)}^2,\\
&\|f\|_{H^1(\R^d)}^2:=\|f\|_{L^2(\R^d)}^2+\|\nabla f\|_{L^2(\R^d)}^2,\ \ 
(f,g)_{L^2(\R^d)}:=\sum_{k=1}^d\int_{\R^d}f^{(k)}(x)\overline{g^{(k)}(x)}dx.
\end{split}
\]
We introduce a new unknown function $U:I_{\max}\times\R^d\rightarrow \C^d\times \C^d\times \C^d$ and a new initial data $U_0:\R^d\rightarrow \C^d\times\C^d\times\C^d$ defined respectively by
\[
U:=(u_1,u_2,u_3),\ \ \ \text{and}\ \ 
U_0:=(u_{1,0},u_{2,0},u_{3,0}).
\]
We define the function space $\mathcal{H}^s:=(H^s(\R^d))^{3d}$ with the norm $\|\cdot\|_{\mathcal{H}^s}$ as
\[
\|U\|_{\mathcal{H}^s}^2:=\|u_1\|_{H^s(\R^d)}^2+\|u_2\|_{H^s(\R^d)}^2+\|u_3\|_{H^s(\R^d)}^2. 
\]

We define the charge $Q:(L^2(\R^d))^{3d}\rightarrow \R_{\ge 0}$ and the energy $E:\mathcal{H}^1\rightarrow \R$ for the system (\ref{NLS}) as
\[
\begin{split}
    Q(U)&:=\|u_1\|_{L^2(\R^d)}^2+\frac{1}{2}\|u_2\|_{L^2(\R^d)}^2+\frac{1}{2}\|u_3\|_{L^2(\R^d)}^2,\\
    E(U)&:=L(U)+N(U),
\end{split}
\]
where $L$ and $N$ are the kinetic energy and the potential energy respectively defined by
\[
\begin{split}
    L(U):=&\frac{\alpha}{2}\|\nabla u_1\|_{L^2(\R^d)}^2+\frac{\beta}{2}\|\nabla u_2\|_{L^2(\R^d)}^2
    +\frac{\gamma}{2}\|\nabla u_3\|_{L^2(\R^d)}^2,\\ 
    N(U):=&{\rm Re}\left(u_3,\nabla (u_1\cdot \overline{u_2})\right)_{L^2(\R^d)}.
\end{split}
\]
Furthermore, we define the momentum $\mathbf{P}:\mathcal{H}^1\rightarrow \R^d$ as 
\[
\mathbf{P}(U):=\mathbf{p}(u_1)+ \mathbf{p}(u_2)+ \mathbf{p}(u_3),
\]
where
\[
\mathbf{p}(f):=(P_1(f),\cdots P_d(f)),\ \ P_k(f):=-\frac{1}{2}{\rm Re}(if,\partial_kf)_{L^2(\R^d)}
\ \ 
(k=1,\cdots d)
\]
for $f=(f_1,\cdots f_d)\in (H^1(\R^d))^d$. 
We note that $Q$, $E$, $L$, $N$ are real valued and well-defined functionals on $\mathcal{H}^1$ 
and $\mathbf{P}$ is vector valued and well-defined functional on $\mathcal{H}^1$. 
In particular, $Q(U)$, $E(U)$, and $\mathbf{P}(U)$ are conserved quantities 
under the flow of (\ref{NLS}). 
Namely, 
if $U$ is a smooth and rapidly decaying solution to (\ref{NLS}), 
then it holds that
\[
Q(U(t))=Q(U_0),\ \ E(U(t))=E(U_0),\ \ 
\mathbf{P}(U(t))=\mathbf{P}(U_0)
\]
for all $t\in I_{\max}$. 
For the conservation law of the charge and the energy, 
see Section\ 7 in \cite{Hirayama14}. 
The conservation law of the momentum 
is implied by the following calculation. 
\[
\begin{split}
\frac{d}{dt}P_k(U)
&=-\sum_{j=1}^3{\rm Re}(i\partial_tu_j,\partial_ku_j)_{L^2(\R^d)}\\
&={\rm Re}(\nabla\cdot u_3,\partial_ku_1\cdot u_2)_{L^2(\R^d)}
+{\rm Re}(\nabla\cdot \overline{u_3},\overline{u_1}\cdot \partial_ku_2)_{L^2(\R^d)}
-{\rm Re}(\nabla(u_1\cdot \overline{u_2}),\partial_ku_3)_{L^2(\R^d)}\\
&=0
\end{split}
\]

For $\omega >0$ and $\mathbf{c}\in \R^d$, we define the functionals $S_{\omega,\mathbf{c}}$, 
$K_{\omega,\mathbf{c}}$, and $L_{\omega,\mathbf{c}}$  on $\mathcal{H}^1$ as
\begin{align}
S_{\omega,\mathbf{c}}(U)&:=E(U)+\omega Q(U)+\mathbf{c}\cdot \mathbf{P}(U),\\
K_{\omega,\mathbf{c}}(U)&:=\partial_{\lambda}S_{\omega,\mathbf{c}}(\lambda U)|_{\lambda =1}=2L(U)+3N(U)+2\omega Q(U)+2 \mathbf{c}\cdot \mathbf{P}(U),\\
L_{\omega,\mathbf{c}}(U)&:=
K_{\omega, \mathbf{c}}(U)-3N(U)=2L(U)+2\omega Q(U)+2\mathbf{c}\cdot \mathbf{P}(U).
\label{Nehari-Nonli}
\end{align}
We note that
\begin{equation}\label{SKL_eq}
S_{\omega,\mathbf{c}}(U)=\frac{1}{2}L_{\omega,\mathbf{c}}(U)+N(U)=\frac{1}{3}K_{\omega,\mathbf{c}}(U)+\frac{1}{6}L_{\omega,\mathbf{c}}(U)
\end{equation}
and
\begin{equation}\label{SKL_eq4}
N(U)=-2S_{\omega,\mathbf{c}}(U)+K_{\omega,\mathbf{c}}(U)
\end{equation}
hold. 
When $d=1$, we write $\mathbf{P}(U)=P(U)$, $\mathbf{c}=c$, 
and $\mathbf{c}\cdot \mathbf{P}(U)=cP(U)$. 

We define the derivatives $D_j$ of the functional $S_{\omega,\mathbf{c}}$ as
\[
\langle D_{j}S_{\omega,\mathbf{c}}(\Phi),\psi\rangle 
:=\lim_{\epsilon \rightarrow 0}\frac{S_{\omega,\mathbf{c}}(\Phi +\epsilon E_j(\psi))-S_{\omega,\mathbf{c}}(\Phi)}{\epsilon},\ \ \Phi \in \mathcal{H}^1,\ \psi \in (H^1(\R^d))^{d},\ j=1,2,3, 
\]
where 
\[
E_1(\psi)=(\psi,\mathbf{0},\mathbf{0}),\ E_2(\psi)=(\mathbf{0},\psi,\mathbf{0}),\ E_3(\psi)=(\mathbf{0},\mathbf{0},\psi)
\in \mathcal{H}^1,\ \ \mathbf{0}=(0,\cdots,0)\in (H^1(\R^d))^d.
\]
By the simple calculation, we can see that
\[
\begin{split}
    \langle D_{1}S_{\omega,\mathbf{c}}(\Phi),\psi\rangle
    &={\rm Re}( -\alpha \Delta \varphi_1+2\omega \varphi_1+i(\mathbf{c}\cdot \nabla)\varphi_1-(\nabla \cdot \varphi_3)\varphi_2,\psi)_{L^2(\mathbb{R}^d)},\\
    \langle D_{2}S_{\omega,\mathbf{c}}(\Phi),\psi\rangle
    &={\rm Re}( -\beta \Delta \varphi_2+\omega \varphi_2+i(\mathbf{c}\cdot \nabla)\varphi_2-(\nabla \cdot \overline{\varphi_3})\varphi_1,\psi)_{L^2(\mathbb{R}^d)},\\
    \langle D_{3}S_{\omega,\mathbf{c}}(\Phi),\psi\rangle
    &={\rm Re}( -\gamma \Delta \varphi_3+\omega \varphi_3+i(\mathbf{c}\cdot \nabla)\varphi_3+\nabla (\varphi_1\cdot \overline{\varphi_2}),\psi)_{L^2(\mathbb{R})^d}
\end{split}
\]
for smooth $\Phi=(\varphi_1,\varphi_2,\varphi_3)$ and smooth $\psi$. 
Therefore, solution $\Phi=(\varphi_1,\varphi_2,\varphi_3)$ to 
the elliptic system
\begin{equation}\label{ellip}
\begin{cases}
    -\alpha \Delta \varphi_1+2\omega \varphi_1+i(\mathbf{c}\cdot \nabla)\varphi_1=(\nabla \cdot \varphi_3)\varphi_2,\\
    -\beta \Delta \varphi_2+\omega \varphi_2+i(\mathbf{c}\cdot \nabla)\varphi_2=(\nabla \cdot \overline{\varphi_3})\varphi_1,\\
    -\gamma \Delta \varphi_3+\omega \varphi_3+i(\mathbf{c}\cdot \nabla)\varphi_3=-\nabla (\varphi_1\cdot \overline{\varphi_2})
\end{cases}
\end{equation}
satisfies
\begin{equation}\label{sta_pt}
D_{1}S_{\omega,\mathbf{c}}(\Phi)=
D_{2}S_{\omega,\mathbf{c}}(\Phi)=
D_{3}S_{\omega,\mathbf{c}}(\Phi)=0.
\end{equation}
\begin{definition}[Weak solution to the stationary problem]
We say that a triple of functions $\Phi=(\varphi_1,\varphi_2,\varphi_3)\in \mathcal{H}^1$ is a weak solution to the stationary problem (\ref{ellip}) if it satisfies
\[
\begin{cases}
    \alpha( \nabla \varphi_1,\nabla \psi)_{L^2(\R^d)}
    +2\omega (\varphi_1,\psi)_{L^2(\R^d)}
    +(i(\mathbf{c}\cdot \nabla)\varphi_1,\psi)_{L^2(\R^d)}
    =((\nabla \cdot \varphi_3)\varphi_2,\psi)_{L^2(\mathbb{R}^d)},\\
    \beta( \nabla \varphi_2,\nabla \psi)_{L^2(\R^d)}
    +\omega (\varphi_2,\psi)_{L^2(\R^d)}
    +(i(\mathbf{c}\cdot \nabla)\varphi_2,\psi)_{L^2(\R^d)}
    =((\nabla \cdot \overline{\varphi_3})\varphi_1,\psi)_{L^2(\mathbb{R}^d)},\\
    \gamma( \nabla \varphi_3,\nabla \psi)_{L^2(\R^d)}
    +\omega (\varphi_3,\psi)_{L^2(\R^d)}
    +(i(\mathbf{c}\cdot \nabla)\varphi_3,\psi)_{L^2(\R^d)}=
    -(\nabla (\varphi_1 \cdot \overline{\varphi_2}),\psi)_{L^2(\mathbb{R}^d)}
\end{cases}
\]
for any $\psi \in (H^1(\R^d))^d$, and then write $S_{\omega,\mathbf{c}}'(\Phi)=0$.
\end{definition}
We note that $\Phi\in \mathcal{H}^1$ becomes a weak solution to (\ref{ellip}) 
if and only if $\Phi$ satisfies (\ref{sta_pt}). 
In the next section, we will show that a weak solution $\Phi \in \mathcal{H}^1$ becomes 
smooth and satisfies the elliptic system (\ref{ellip}) 
under the suitable condition for $\omega$ and $\mathbf{c}$ ($\omega >\frac{\sigma|\mathbf{c}|^2}{4}$). 
(See Corollary~\ref{phi_regularity} and Remark~\ref{weak_classi_sol} below.) 
We define the sets of weak solutions to (\ref{ellip}) as 
\[
\begin{split}
    \mathcal{E}_{\omega,\mathbf{c}}
    &:=\{\Psi \in \mathcal{H}^1\backslash \{(\mathbf{0},\mathbf{0},\mathbf{0})\}\ |\ \ 
    S_{\omega,\mathbf{c}}'(\Psi)=0\},\\
    \mathcal{G}_{\omega,\mathbf{c}}
    &:=\{\Psi \in \mathcal{E}_{\omega,\mathbf{c}}\ |\ \ 
    S_{\omega,\mathbf{c}}(\Psi)\le S_{\omega,\mathbf{c}}(\Theta)\ \text{for\ any}\ 
    \Theta \in \mathcal{E}_{\omega,\mathbf{c}}\}. 
\end{split}
\]
We call an element $\Psi \in \mathcal{G}_{\omega,\mathbf{c}}$ a ground state to the elliptic system (\ref{ellip}).
\begin{remark}\label{EK_rem}
If $\Psi \in \mathcal{E}_{\omega,\mathbf{c}}$, 
then it holds $K_{\omega,\mathbf{c}}(\Psi)=0$. 
Indeed, if $\Psi =(\psi_1,\psi_2,\psi_3)\in \mathcal{E}_{\omega,\mathbf{c}}$, then we have
\[
K_{\omega,\mathbf{c}}(\Psi )=\partial_{\lambda}S_{\omega,\mathbf{c}}(\lambda \Psi)\big|_{\lambda =1}
=\sum_{j=1}^3\langle D_jS_{\omega,\mathbf{c}}(\Psi),\psi_j\rangle =0. 
\]
\end{remark}
\begin{remark}\label{sol_scale}
For $\Psi \in \mathcal{H}^1$, we set $\Psi_{\omega}$ for $\omega>0$ as 
\[
\Psi_{\omega}(x):=\omega^{\frac{1}{2}}\Psi (\omega^{\frac{1}{2}}x). 
\]
By a simple calculation, we can see that
$\Psi \in \mathcal{E}_{1,\frac{\mathbf{c}}{\sqrt{\omega}}}$ is equivalent to 
$\Psi_{\omega}\in \mathcal{E}_{\omega,\mathbf{c}}$ and it holds that
\[
Q(\Psi_{\omega})=\omega^{1-\frac{d}{2}}Q(\Psi),\ \ 
E(\Psi_{\omega})=\omega^{2-\frac{d}{2}}E(\Psi),\ \ 
\mathbf{P}(\Psi_{\omega})=\omega^{\frac{3-d}{2}}\mathbf{P}(\Psi).
\]
These identities imply the equality
\[
S_{\omega,\mathbf{c}}(\Psi_{\omega})=\omega^{2-\frac{d}{2}}S_{1,\frac{\mathbf{c}}{\sqrt{\omega}}}(\Psi). 
\]
\end{remark}
\begin{remark}
For $\theta \in \R$, we define an unitary operator $\Lambda (\theta)$ as
\[
\Lambda (\theta)\Phi:=(e^{2i\theta}\varphi_1, e^{i\theta}\varphi_2, e^{i\theta}\varphi_3).  
\]
If we set $U(t,x)=U_{\omega,\mathbf{c}}(t,x):=\Lambda (\omega t)\Phi (x-\mathbf{c}t)$ 
for $\Phi=(\phi_1,\phi_2,\phi_3)\in \mathcal{E}_{\omega,\mathbf{c}}$,  
then $U$ is a solution to the system (\ref{NLS}) with $U_0(x)=\Phi(x)$. 
We call this solution $U$ a 
two-parameter solitary wave solution. 
\end{remark}
We study the minimization problem
\[
\mu_{\omega ,\mathbf{c}}:=
\inf \{S_{\omega,\mathbf{c}}(\Psi)\ |\ \Psi \in \mathcal{H}^1\backslash \{(\mathbf{0},\mathbf{0},\mathbf{0})\},\ 
K_{\omega ,\mathbf{c}}(\Psi )=0\}
\]
and define the set of minimizers
\[
\begin{split}
\mathcal{M}_{\omega,\mathbf{c}}:=
\{\Psi \in \mathcal{H}^1\backslash \{(\mathbf{0},\mathbf{0},\mathbf{0})\}\ |\  
S_{\omega,\mathbf{c}}(\Psi)=\mu_{\omega,\mathbf{c}},\ 
K_{\omega,\mathbf{c}}(\Psi)=0\}. 
\end{split}
\]
We note that if $S_{\omega,\mathbf{c}}(\Psi)<\mu_{\omega,\mathbf{c}}$ 
and $K_{\omega,\mathbf{c}}(\Psi)=0$ hold, 
then $\Psi =(\mathbf{0},\mathbf{0},\mathbf{0})$ by the definition of $\mu_{\omega,\mathbf{c}}$. 
Furthermore, for $\eta >0$, we define a subset of $\mathcal{M}_{\omega,\mathbf{c}}$ as
\[
\begin{split}
\mathcal{M}_{\omega,\mathbf{c}}^*(\eta):=
\{\Psi \in \mathcal{M}_{\omega,\mathbf{c}}\ |\ G(\Psi)\ge \eta\},
\end{split}
\]
where $G:\mathcal{H}^1\rightarrow \R$ is defined by
\begin{equation}
\label{defG}
G(\Psi):=(4-2d)\omega Q(\Psi)+(3-d)\mathbf{c}\cdot \mathbf{P}(\Psi).
\end{equation}
Then
\[
\mathcal{M}_{\omega,c}^*(\eta)
=
\begin{cases}
    \{\Psi \in \mathcal{M}_{\omega,c}\ |\ 
\omega Q(\Psi)+cP(\Psi)\ge \eta\},&\text{if}\ d=1,\\
\{\Psi \in \mathcal{M}_{\omega,\mathbf{c}}\ |\ 
\mathbf{c}\cdot \mathbf{P}(\Psi)\ge \eta\},&\text{if}\ d=2,\\
\emptyset,&\text{if}\ d=3.
\end{cases}
\] 

We introduce a positive number $\sigma$ given by
\[
\sigma :=\max\left\{\frac{1}{2|\alpha|},\frac{1}{|\beta|},\frac{1}{|\gamma|}\right\}
=\left(\min\{2|\alpha|,|\beta|,|\gamma|\}\right)^{-1}>0
\]
and define
\[
\R^3_+:=\left\{(\alpha,\beta,\gamma)\in \R^3\ |\ \alpha >0,\ \beta>0,\ \gamma>0\right\},\ \ 
\R^3_-:=\left\{(\alpha,\beta,\gamma)\in \R^3\ |\ \alpha <0,\ \beta<0,\ \gamma<0\right\}. 
\]
We note that $(\alpha,\beta,\gamma)\in \R^3_+\cup \R^3_-$, 
then the kinetic energy $L$ becomes positive or negative definite. 
The main results in the present paper are the followings. 
\begin{theorem}\label{GS_exist_th}
Let $d\in \{1,2,3\}$ and $(\alpha,\beta,\gamma)\in \R^3_+\cup \R^3_-$. 
We assume that  
$(\omega,\mathbf{c})\in \R\times \R^d$ 
satisfies
$\omega >\frac{\sigma |\mathbf{c}|^2}{4}$. 
Then we have 
$\mathcal{G}_{\omega,\mathbf{c}}=\mathcal{M}_{\omega,\mathbf{c}}\neq \emptyset$. 
Namely, there exists at least one of ground state. 
\end{theorem}
According to Theorem~\ref{NLS_LWP_Known}, 
the local well-posedness of (\ref{NLS}) in $\mathcal{H}^1$
is obtained under the condition $(\alpha-\gamma)(\beta+\gamma)\ne 0$. 
We note that $\beta+\gamma =0$ does not occur for $(\alpha,\beta,\gamma)\in \R^3_+\cup \R^3_-$. 
\begin{theorem}[Global well-posedness below the ground state level and uniform boundedness]\label{gwp}
Let $d\in \{1,2,3\}$ and 
$(\alpha,\beta,\gamma)\in \R^3_+\cup \R^3_-$. 
We assume that $\alpha -\gamma \ne 0$ and 
$(\omega,\mathbf{c})\in \R\times \R^d$ 
satisfies
$\omega >\frac{\sigma |\mathbf{c}|^2}{4}$.
If $U_0\in \mathcal{H}^1$ satisfies
\[
S_{\omega,\mathbf{c}}(U_0)<\mu_{\omega,\mathbf{c}},\ \ K_{\omega,\mathbf{c}}(U_0)>0,
\]
then $I_{\max}=\R$, that is, the local solution $U(t)\in \mathcal{H}^1$ to (\ref{NLS}) with $U(0)=U_0$ 
can be extended globally in time. Moreover for any $t\in \R$, the estimate $S_{\omega,\mathbf{c}}(U(t))<\mu_{\omega,\mathbf{c}}$ holds.
\end{theorem}
When $d=2$, which is $L^2$-critical case, 
the above global well-posedness result depending on the parameters $\omega, \mathbf{c}$ can be written as the following form without the parameters.
\begin{corollary}[Global well-posedness below the ground state without parameters in $2d$]
\label{gwp_2d}
Let $d=2$ and 
$(\alpha,\beta,\gamma)\in \R^3_+\cup \R^3_-$ satisfies $\alpha -\gamma \ne 0$. 
If $U_0\in \mathcal{H}^1$ satisfies 
\begin{equation}\label{gwp_2d_cond}
Q(U_0)<Q(\Phi)-E(\Phi)
\end{equation}
for some $\Phi\in \mathcal{M}_{1,\mathbf{c}_0}$ 
and some $\mathbf{c}_0\in \R^2$
with $|\mathbf{c}_0|<\frac{2}{\sqrt{\sigma}}$, namely
\[
Q(U_0)<\sup_{|\mathbf{c}_0|<\frac{2}{\sqrt{\sigma}}}
\sup_{\Phi\in \mathcal{M}_{1,\mathbf{c}_0}}\left(Q(\Phi)-E(\Phi)\right),
\]
then the local solution $U(t)\in \mathcal{H}^1$ to (\ref{NLS}) with $U(0)=U_0$ 
can be extended globally in time.
\end{corollary}
\begin{remark}
    When $d=2$, the estimate $Q(\Phi)-E(\Phi)>0$ holds for $\Phi \in \mathcal{M}_{1,\mathbf{c}_0}$. This can be seen as follows.
    For $\Phi \in \mathcal{M}_{1,\mathbf{c}_0}$, the estimate
    \[
    \mathbf{c}_0\cdot \mathbf{P}(\Phi)=-2E(\Phi)
    \]
    holds by Proposition~\ref{LN_iden} below. 
    Therefore, we have
    \[
    0<\mu_{1,\mathbf{c}_0}
    =E(\Phi)+Q(\Phi)+\mathbf{c}_0\cdot \mathbf{P}(\Phi)
    =Q(\Phi)-E(\Phi)
    \]
    because $\mu_{1,\mathbf{c}_0}>0$ by Proposition~\ref{mu_prop} below.
\end{remark}
\begin{remark}
If $\Phi \in \mathcal{M}_{1,\mathbf{0}}$, then
the condition (\ref{gwp_2d_cond}) is equivalent to 
$Q(U_0)<Q(\Phi)$ 
because $\Phi \in \mathcal{M}_{1,\mathbf{0}}$ with $d=2$ 
satisfies $E(\Phi)=0$ by 
Proposition~\ref{LN_iden} below.
\end{remark}
\begin{remark}
Because the smallness condition for initial data 
is not necessary except (\ref{gwp_2d_cond}) 
in Corollary~\ref{gwp_2d}, 
this result is the improvement 
of global result in Theorem~\ref{NLS_LWP_Known} for $d=2$. 
(See, also Remark~\ref{gwp_known_rem}.)
\end{remark}
Next we state our result of orbital stability of the set $\mathcal{M}_{\omega,\mathbf{c}}^*(\eta)$ for arbitrary $\eta>0$, which solves one of the open problems proposed in \cite[P301]{CC04}.
\begin{theorem}[Stability of the set $\mathcal{M}_{\omega,\mathbf{c}}^*(\eta)$]
\label{stability}
Let $d\in \{1,2\}$ and $(\alpha,\beta,\gamma)\in \R^3_+\cup \R^3_-$. 
We assume that $\alpha -\gamma \ne 0$ and 
$(\omega,\mathbf{c})\in \R\times \R^d$ 
satisfies
$\omega >\frac{\sigma |\mathbf{c}|^2}{4}$. 
For $\eta >0$, if
$\mathcal{M}_{\omega,\mathbf{c}}^*(\eta)\ne \emptyset$ 
holds,  
then the set $\mathcal{M}_{\omega,\mathbf{c}}^*(\eta)$ is 
orbitally stable. 
More precisely, for any $\epsilon >0$, 
there exists $\delta >0$ such that 
if $U_0\in \mathcal{H}^1$ satisfies
\[
\inf_{\Phi \in \mathcal{M}_{\omega,\mathbf{c}}^*(\eta)}\|U_0-\Phi\|_{\mathcal{H}^1}<\delta,
\]
then the solution $U(t)$ to {\rm (\ref{NLS})} with $U(0)=U_0$ exists globally in time 
and satisfies
\[
\sup_{t\ge 0}\inf_{\Phi \in \mathcal{M}_{\omega,\mathbf{c}}^*(\eta)}\|U(t)-\Phi\|_{\mathcal{H}^1}<\epsilon. 
\]
\end{theorem}
Next we give orbital stability of the set $\mathcal{M}_{\omega,c}$ 
for $1$-dimensional setting
as a Corollary of Theorem~\ref{stability}. 
\begin{corollary}\label{stab_1d_result}
Let $d=1$. Then the ground-state set $\mathcal{M}_{\omega,c}$ is orbitally stable if $|c|$ is small enough. 
\end{corollary}
To prove the stability of the ground-states set, 
we will use the variational argument for 
the function $\mu(\omega,\mathbf{c}):=\mu_{\omega,\mathbf{c}}$ with respect to $(\omega, \mathbf{c})$, 
whose argument is used in \cite{CO06}. 
In this argument, the second derivatives of $\mu$ play an important role. 
However in our problem, there is a difficulty that the ground state $\Phi\in \mathcal{M}_{\omega,\mathbf{c}}$ 
cannot be written explicitly. 
Therefore, we cannot calculate the second derivatives of $\mu$ specifically. 
To avoid this difficulty, we use the argument by using a so-called scaling curve, which was used by Hayashi in \cite{Hayashi22}. Here the scaling curve is given as a map $\tau \mapsto \left((\sqrt{\omega}-\tau)^2,\frac{\mathbf{c}}{\sqrt{\omega}}
(\sqrt{\omega} -\tau)\right)$. 
By considering the restriction of $\mu$ on the scaling curve, 
we can treat $\mu$ such as a polynomial function 
and calculate the second derivative. 
Thanks to these properties and the Pohozaev's identity (Proposition~\ref{LN_iden}) below, 
we can give the sufficient condition for ground states to establish the stability. 
Furthermore, by checking the sufficient condition when $d=1$ (see, Proposition~\ref{1d_stability}), 
we can obtain the stability of the ground-states set $\mathcal{M}_{\omega,\mathbf{c}}$ 
for small $\mathbf{c}$. 
In particular, our results contain the stability of the set of traveling waves. 

From the next section, 
we always consider the case $(\alpha,\beta,\gamma)\in \R^3_+$ 
because we can assume that $\alpha >0$ without loss of 
generality. 
Indeed, if $(u_1,u_2,u_3)$ is a solution to (\ref{NLS}), 
then $(\widetilde{u}_1,\widetilde{u}_2,\widetilde{u}_3)$ defined by 
$\widetilde{u}_j(t,x):=u_j(-t,-x)$\ $(j=1,2,3)$ is a solution to
\[
\begin{cases}
i\partial_t\widetilde{u}_1-\alpha \Delta \widetilde{u}_1=-(\nabla \cdot \widetilde{u}_3)\widetilde{u}_2,&-t\in I_{\max},\ x\in \R^d,\\
i\partial_t\widetilde{u}_2-\beta \Delta \widetilde{u}_2=-(\nabla \cdot \overline{\widetilde{u}_3})\widetilde{u}_1,&-t\in I_{\max},\ x\in \R^d,\\
i\partial_t\widetilde{u}_3-\gamma \Delta \widetilde{u}_3=\nabla (\widetilde{u}_1\cdot \overline{\widetilde{u}_2}),&-t\in I_{\max},\ x\in \R^d,\\
(\widetilde{u}_1(0,x),\widetilde{u}_2(0,x),\widetilde{u}_3(0,x))=(u_{1,0}(-x),u_{2,0}(-x),u_{3,0}(-x)), &x\in \R^d.
\end{cases}    
\]
If $\alpha<0$, then it suffices to consider this Cauchy problem instead of (\ref{NLS}). 

We introduce similar results such as existence of ground state and stability of solitary wave for other Schr\"odinger type equations. Guo-Wu \cite{GW95} and Colin-Ohta \cite{CO06} studied orbital stability of a two-parameter family of solitary waves $\{u_{\omega,c}\}_{(\omega,c)\in \R\times\R}$ to the single derivative nonlinear Schr\"odinger equation in one spatial dimension
\begin{equation}
\label{sdnls}
i\partial_tu +\partial_x^2u+i\partial_x(|u|^2u)=0,\ \ \ (t,x)\in \R\times \R,
\end{equation}
where $u_{\omega,c}$ for $(\omega,c)\in \R\times\R$ is the form of
\[
u_{\omega,c}(t,x):=\phi_{\omega,c}(x-ct)\exp\left\{i\omega t+i\frac{c}{2}(x-ct)-\frac{3i}{4}\int_{-\infty}^{x-ct}|\phi_{\omega,c}(\eta)|^2d\eta\right\}.
\]
Here $\phi_{\omega,c}$ is the unique ground state of the elliptic equation
\[
-\partial_{x}^2\phi+\left(\omega-\frac{c^2}{4}\right)\phi+\frac{c}{2}\phi^3-\frac{3}{16}\phi^5=0.
\]
If $c^2<4\omega$, then $\phi_{\omega,c}$ shows an exponential decay and is explicitly written as
\[
\phi_{\omega,c}(x):=\left[\frac{\sqrt{\omega}}{4\omega-c^2}\left\{\cosh (\sqrt{(4\omega-c^2)x})-\frac{c}{\sqrt{4\omega}}\right\}\right]^{-1/2}
\]
and the charge of $\phi_{\omega,c}$ is given by
\[
\|\phi_{\omega,c}\|_{L^2(\R)}^2=8\tan^{-1}\sqrt{\frac{\sqrt{4\omega}+c}{\sqrt{4\omega}-c}}<4\pi.
\]
The orbital stability of those solitons was proved in \cite{GW95} for $c<0$ and $c^2<4\omega$ and in \cite{CO06} for any $c^2<4\omega$. See Kwon-Wu \cite{KW18} for the endpoint case $c^2=4\omega$. Fukaya-Hayashi-Inui \cite{FHI17} gave a sufficient condition for global existence of solutions to the generalized version of (\ref{sdnls}) by a variational argument (see also \cite{W15, BP22}). For the system of nonlinear Schr\"odinger equations without derivative nonlinearities, see \cite{A18, CCO06, KO22, FHI23} and their references. 
In particular, in \cite{FHI23}, 
similar results as Theorems~\ref{GS_exist_th} 
and ~\ref{gwp} are obtained for 
the following system of Schr\"odinger equations:
\begin{equation}\label{two_NLS}
\begin{cases}
i\partial_tu+\frac{1}{2m}\Delta u=\overline{u}v,\\
i\partial_tv+\frac{1}{2M}\Delta v=u^2,
\end{cases}
(t,x)\in \R\times \R^d,
\end{equation}
where $m,M\in \R\backslash\{0\}$ are constants. In \cite{FHI23}, the authors also proved 
the global well-posedness of (\ref{two_NLS}) with $d=4$ 
(which is $L^2$-critical case)
for arbitrary large initial data with modification of oscillations.  

We use the shorthand $A\lesssim B$ to denote the estimate $A\leq CB$ with
some constant $C>0$. The notation $A\sim B$ stands for $A\lesssim B$
and $B\lesssim A$. 
We give a notation table for reader's convenience. 
\begin{table}[h]
\caption{Notation table}\label{tab1}
\vspace{.3cm}
\renewcommand{\arraystretch}{1.1}
 \begin{tabularx}{\linewidth}{|c|X|}
\hline
$\mathcal{H}^s$& $s$-th order Sobolev space $(H^s(\R^d))^{3d}$\\
$Q(U)$ & Charge $\|u_1\|_{L^2(\R^d)}^2+\frac{1}{2}\|u_2\|_{L^2(\R^d)}^2+\frac{1}{2}\|u_3\|_{L^2(\R^d)}^2$\\
$L(U)$ & Kinetic energy $\frac{\alpha}{2}\|\nabla u_1\|_{L^2(\R^d)}^2+\frac{\beta}{2}\|\nabla u_2\|_{L^2(\R^d)}^2
    +\frac{\gamma}{2}\|\nabla u_3\|_{L^2(\R^d)}^2$\\
$N(U)$ & Potential energy ${\rm Re}\left(u_3,\nabla (u_1\cdot \overline{u_2})\right)_{L^2(\R^d)}$\\
$\mathbf{P}(U)$ & Momentum $(P_1(u_1),\cdots,P_d(u_1))$+$(P_1(u_2),\cdots,P_d(u_2))$+$(P_1(u_3),\cdots,P_d(u_3))$\\
$P_k(u_j)$ & $k$-th component of the momentum 
$-\frac{1}{2}\sum_{j=1}^3{\rm Re}(iu_j,\partial_ku_j)_{L^2(\R^d)}$\\
$\sigma$ & $\sigma:=\max\left\{\frac{1}{2|\alpha|},\frac{1}{|\beta|},\frac{1}{|\gamma|}\right\}=\min(2|\alpha|,|\beta|,|\gamma|)^{-1}>0$\\
$\omega$ & Frequency parameter satisfying $\omega>\frac{\sigma|\mathbf{c}|^2}{4}$\\
$\mathbf{c}$ & Propagation speed satisfying $\omega>\frac{\sigma|\mathbf{c}|^2}{4}$\\
$S_{\omega,\mathbf{c}}(U)$ & Action $E(U)+\omega Q(U)+\mathbf{c}\cdot \mathbf{P}(U)$\\
$K_{\omega,\mathbf{c}}(U)$ & Nehari functional $\frac{d}{d\lambda}S_{\omega,\mathbf{c}}(\lambda U)\big|_{\lambda=1}$\\
$L_{\omega, \mathbf{c}}(U)$ & Nehari functional without nonlinearity $K_{\omega, \mathbf{c}}(U)-3N(U)$\\
$\mathcal{G}_{\omega,\mathbf{c}}$ & Set of ground state $\{\Psi \in \mathcal{E}_{\omega,\mathbf{c}}\ |\ \ 
    S_{\omega,\mathbf{c}}(\Psi)\le S_{\omega,\mathbf{c}}(\Theta)\ \text{for\ any}\ 
    \Theta \in \mathcal{E}_{\omega,\mathbf{c}}\}$\\
$\mu_{\omega ,\mathbf{c}}$& Mountain pass level $\inf \{S_{\omega,\mathbf{c}}(\Psi)\ |\ \Psi \in \mathcal{H}^1\backslash \{(\mathbf{0},\mathbf{0},\mathbf{0})\},\ 
K_{\omega ,\mathbf{c}}(\Psi )=0\}$\\
$\mathcal{M}_{\omega,\mathbf{c}}$ & Set of minimizers $\{\Psi \in \mathcal{H}^1\backslash \{(\mathbf{0},\mathbf{0},\mathbf{0})\}\ |\  
S_{\omega,\mathbf{c}}(\Psi)=\mu_{\omega,\mathbf{c}},\ 
K_{\omega,\mathbf{c}}(\Psi)=0\}$\\
\hline
 \end{tabularx}
\renewcommand{\arraystretch}{1} 
\end{table}
%
%
\section{Properties of weak solutions to the stationary problem}

In this section we study properties of weak solutions (not necessarily ground state) to stationary problem.
\begin{proposition}\label{GS_regu}
Let $d\in \{1,2,3\}$ and $\alpha,\beta,\gamma>0$. 
Assume that $(\omega,\mathbf{c})\in \R\times \R^d$ 
satisfies $\omega >\frac{\sigma|\mathbf{c}|^2}{4}$. 
If $\Phi =(\varphi_1,\varphi_2,\varphi_3)\in \mathcal{H}^1$ be a weak solution to {\rm (\ref{ellip})}, 
then $\Phi \in (W^{2,p}(\R^d))^{3d}$ holds for any $p\in [1,2]$, where $W^{2,p}(\R^d)$ denotes the second order $L^p$-based Sobolev space, that is, $W^{2,p}(\R^d):=\{f\in L^p(\R^d)\ |\ \sum_{|\alpha|\le 2}\|\partial_x^{\alpha}f\|_{L^p}<\infty\}$. 
In particular, $\varphi_j \in (H^2(\R^d))^d$ $(j=1,2,3)$ holds. 
\end{proposition}
\begin{proof}
We set $p_0:=\max\{1,\frac{d}{2}\}$ and $q_0:=2p_0/(2-p_0)$. 
Then, we have
\[
\frac{1}{p_0}=\frac{1}{2}+\frac{1}{q_0}. 
\]
We note that $1\le p_0<2$, $q_0\ge 2$ and
\[
1\ge \frac{d}{2} -\frac{d}{q_0}
\]
hold because $d\le 3$. 
By the H\"older inequality and the Sobolev inequality, 
we have
\[
\|(\nabla \cdot \varphi_3)\varphi_2\|_{L^{p_0}}
\le \|\nabla \cdot \varphi_3\|_{L^2}\|\varphi_2\|_{L^{q_0}}
\lesssim \|\varphi_3\|_{H^1}\|\varphi_2\|_{H^1}<\infty. 
\]
By the first equation of (\ref{ellip}), we obtain
\[
\|\varphi_1\|_{W^{2,p_0}}=
\|(-\alpha \Delta +2\omega +i\mathbf{c}\cdot \nabla )^{-1}
\{(\nabla \cdot \varphi_3)\varphi_2\}\|_{W^{2,p_0}}
\lesssim \|(\nabla \cdot \varphi_3)\varphi_2\|_{L^{p_0}}<\infty
\]
because
\[
\alpha |\xi|^2+2\omega -\mathbf{c}\cdot \xi
=\alpha \left|\xi -\frac{\mathbf{c}}{2\alpha}\right|^2+2\left(\omega -\frac{|\mathbf{c}|^2}{8\alpha}\right) 
\sim 1+|\xi|^2
\]
holds when $\omega >\frac{\sigma |\mathbf{c}|^2}{4}$. 
Therefore, we have $\varphi_1\in (W^{2,p_0}(\R^d))^d$. 
By the same argument, we also obtain 
$\varphi_2, \varphi_3\in (W^{2,p_0}(\R^d))^d$. 
Since $p_0\ge \frac{d}{2}$, 
the relation
$W^{2,p_0}(\R^d)\subset L^q(\R^d)$ 
holds for any $q\in [2,\infty]$ 
by the Sobolev embedding theorem. 
This implies
 $\Phi \in (L^q(\R^d))^{3d}$ 
for any $q\in [2,\infty]$. 

For any $p\in [1,2]$, we put
\[
\frac{1}{q}:=\frac{1}{p}-\frac{1}{2}. 
\]
We note that $q\ge 2$. 
Therefore, we have $\varphi_2 \in L^q(\R^d)$ and
$
\|\varphi_1\|_{W^{2,p}}
\lesssim
\|(\nabla \cdot \varphi_3)\varphi_2\|_{L^{p}}
<\infty
$
by the same argument as above. 
Similarly, we also have 
$
\|\varphi_2\|_{W^{2,p}}<\infty,\ \ 
\|\varphi_3\|_{W^{2,p}}<\infty. 
$
\end{proof}
\begin{corollary}\label{phi_regularity}
Let $d\in \{1,2,3\}$ and $\alpha,\beta,\gamma>0$. 
Assume that $(\omega,\mathbf{c})\in \R\times \R^d$ 
satisfies $\omega >\frac{\sigma |\mathbf{c}|^2}{4}$. 
If $\Phi =(\varphi_1,\varphi_2,\varphi_3)\in \mathcal{H}^1$ be a weak solution to {\rm (\ref{ellip})}, 
then it holds 
\[
\varphi_1,\ \varphi_2,\ \varphi_3\in 
\left(\bigcap_{m=1}^{\infty}H^m(\R^d)\right)^d.  
\]
\end{corollary}
\begin{proof}
We first prove that 
if $\varphi_j\in (H^m(\R^d))^d$ $(j=1,2,3)$
 for some $m\in \N$ with $m\ge 2$, 
then $\varphi_j \in (H^{m+1}(\R^d))^d$ $(j=1,2,3)$ 
holds. 
We note that 
$H^m(\R^d)$ is Banach algebra 
since $d<4$ and $m\ge 2$. 
Therefore, by the third equation of (\ref{ellip}) and 
the H\"older inequality, we have
\[
\begin{split}
\|\partial_k \varphi_3\|_{H^m}
&=\|\partial_k (-\gamma \Delta +\omega +i\mathbf{c}
\cdot \nabla)^{-1}\nabla (\varphi_1\cdot \overline{\varphi_2})\|_{H^m}
\lesssim \|\varphi_1 \cdot \varphi_2\|_{H^m}
\lesssim \|\varphi_1\|_{H^m}\|\varphi_2\|_{H^m}
<\infty
\end{split}
\]
for $k=1,\cdots, d$ because $\omega >\frac{\sigma |\mathbf{c}|^2}{4}$. 
This means that $\varphi_3\in (H^{m+1}(\R^d))^d$. 
Next, the first equation of (\ref{ellip})
and the H\"older inequality, we have
\[
\begin{split}
\|\partial_k \varphi_1\|_{H^m}
&=\|\partial_k (-\alpha \Delta +2\omega +i\mathbf{c}
\cdot \nabla)^{-1}\{(\nabla\cdot \varphi_3)\varphi_2\}\|_{H^m}
\lesssim \|(\nabla\cdot \varphi_3)\varphi_2\|_{H^m}
\lesssim \|\varphi_3\|_{H^{m+1}}\|\varphi_2\|_{H^m}
<\infty
\end{split}
\]
for $k=1,\cdots,d$ because $\omega >\frac{\sigma |\mathbf{c}|^2}{4}$. 
This means that $\varphi_1\in (H^{m+1}(\R^d))^d$. 
Similarly, we also have $\varphi_2\in (H^{m+1}(\R^d))^d$. 

By Proposition~\ref{GS_regu}, 
we have $\varphi_1$, $\varphi_2, \varphi_3\in (H^2(\R^d))^d$. Therefore, we obtain $\varphi_1$, $\varphi_2, \varphi_3\in (H^m(\R^d))^d$ for any $m\in \N$ by the induction. 
\end{proof}
\begin{remark}\label{weak_classi_sol}
By Corollary~\ref{phi_regularity}, 
the weak solution $\Phi =(\varphi_1,\varphi_2,\varphi_3)\in \mathcal{H}^1$ is also a classical solution 
and derivatives of $\Phi$ vanish at infinity
because the embedding
\[
H^{m}(\R^d)
\subset \left\{f\in C^r(\R^d)\left|\ \displaystyle \lim_{|x|\rightarrow \infty}\partial^{\alpha}f(x)=0,\ \alpha \in \Z^d:\text{multi\ index\ with}\ |\alpha|\le r\right.\right\}
\]
holds for $m>\frac{d}{2}+r$. 
In particular, $\varphi_3$ satisfies
\begin{equation}\label{phi_3_vanish_infi}
\lim_{|x|\rightarrow \infty}\nabla \cdot \varphi_3(x)=0.
\end{equation}
\end{remark}
The following proposition (Pohozaev's identity) plays an important role to prove 
the stability in Section~\ref{GWP_Sta}. 
\begin{proposition}[Pohozaev's identity]\label{LN_iden}
Let $d\in \{1,2,3\}$ and $\alpha,\beta,\gamma>0$. 
Assume that $(\omega,\mathbf{c})\in \R\times \R^d$ 
satisfies $\omega >\frac{\sigma |\mathbf{c}|^2}{4}$.
If $\Phi=(\varphi_1,\varphi_2,\varphi_3)\in \mathcal{H}^1$ 
is a weak solution to {\rm (\ref{ellip})}, 
then $\Phi$ satisfies
\[
2L(\Phi)+\left(\frac{d}{2}+1\right)N(\Phi)
+\mathbf{c}\cdot \mathbf{P}(\Phi)=0. 
\]
\end{proposition}
\begin{proof}
For $\lambda >0$, we set
\begin{equation}\label{scale_l2inv}
\Phi^{\lambda}(x):=\lambda^{\frac{d}{2}}\Phi (\lambda x). 
\end{equation}
Then, we have
\[
\frac{d}{d\lambda}S_{\omega,\mathbf{c}}(\Phi^{\lambda})\big|_{\lambda=1}
=\sum_{j=1}^3\left.\left\langle (D_jS_{\omega,\mathbf{c}})(\Phi^{\lambda}), 
\frac{d}{d\lambda}\Phi^{\lambda}\right\rangle\right|_{\lambda =1}
=\sum_{j=1}^3\left\langle D_jS_{\omega,\mathbf{c}}(\Phi), \frac{d}{2}\Phi+\nabla \cdot \Phi\right\rangle. 
\]
Because $\Phi$ is a weak solution to (\ref{ellip}), 
we have $\varphi_1$, $\varphi_2$, $\varphi_3\in (H^2(\R^d))^d$ 
by Proposition~\ref{GS_regu}. 
Namely, $\frac{d}{2}\Phi +\nabla \cdot \Phi \in \mathcal{H}^1$ holds. 
Therefore, 
we obtain
\[
\frac{d}{d\lambda}S_{\omega,\mathbf{c}}(\Phi^{\lambda})\big|_{\lambda=1}
=0. 
\]
On the other hand, we can easily see that
\[
L(\Phi^{\lambda})=\lambda^2L(\Phi),\ \ 
N(\Phi^{\lambda})=\lambda^{\frac{d}{2}+1}N(\Phi),\ \ 
Q(\Phi^{\lambda})=Q(\Phi),\ \ 
\mathbf{P}(\Phi^{\lambda})=\lambda\mathbf{P}(\Phi). 
\]
It implies that
\[
\begin{split}
\frac{d}{d\lambda}S_{\omega,\mathbf{c}}(\Phi^{\lambda})|_{\lambda=1}
&=\frac{d}{d\lambda}\left.\left(L(\Phi^{\lambda})+N(\Phi^{\lambda})+\omega Q(\Phi^{\lambda})
+\mathbf{c}\cdot \mathbf{P}(\Phi^{\lambda})\right)\right|_{\lambda =1}\\
&=2L(\Phi)+\left(\frac{d}{2}+1\right)N(\Phi)
+\mathbf{c}\cdot \mathbf{P}(\Phi)
\end{split}
\]
and we get the desired identity. 
\end{proof}
\section{Existence of a ground state to the stationary problem}
In this section we give a proof of existence of a ground state to (\ref{ellip}), which addresses a open problem proposed by Colin-Colin \cite{CC04}.

We prove coercivity of the functional $L_{\omega, \mathbf{c}}$ given by (\ref{Nehari-Nonli}), which plays an important role to prove Proposition \ref{mu_prop} and Lemma \ref{1d_stab_lemm} below. 
\begin{proposition}[Coecivity of the functional $L_{\omega,\mathbf{c}}$]\label{L_bdd}
Let $d\in \{1,2,3\}$ and $\alpha,\beta,\gamma>0$. 
Assume that $(\omega,\mathbf{c})\in \R\times \R^d$ 
satisfies $\omega >\frac{\sigma |\mathbf{c}|^2}{4}$.
Then, there exists a constant $C=C_{\omega}>0$ independent of $\mathbf{c}$ such that the estimate
\[
L_{\omega, \mathbf{c}}(U)\ge 
C(4\omega-\sigma |\mathbf{c}|^2)\|U\|_{\mathcal{H}^1}^2 
\]
holds for any $U\in \mathcal{H}^1$, where $L_{\omega, \mathbf{c}}$ is defined by (\ref{Nehari-Nonli}).
\end{proposition}
\begin{proof}
Let $U=(u_1,u_2,u_3)\in \mathcal{H}^1$ and $\mathbf{c}=(c_1,\cdots ,c_d)\in \R^d$. 
For parameters $A_1,A_2,A_3>0$, we set
\[
I_{j,k}:=\frac{1}{A_j}
\left\|A_j\partial_ku_j-\frac{c_k}{4}iu_j\right\|_{L^2(\R^d)}^2\ \ 
(j=1,2,3,\ k=1,\cdots, d). 
\]
By a direct computation, for $j\in \{1,2,3\}$, the following identities hold:
\[
\mathbf{c}\cdot \mathbf{p}(u_j)
=-\frac{1}{2}\sum_{k=1}^dc_k{\rm Re}(iu_j,\partial_ku_j)_{L^2(\R^d)}
=\sum_{k=1}^dI_{j,k}-A_j\|\nabla u_j\|_{L^2(\R^d)}^2
-\frac{|\mathbf{c}|^2}{16A_j}\|u_j\|_{L^2(\R^d)}^2.
\]
This implies that by (\ref{Nehari-Nonli}), the following holds:
\[
\begin{split}
L_{\omega, \mathbf{c}}(U)
&=(\alpha -2A_1)\|\nabla u_1\|_{L^2(\R^d)}^2
+(\beta -2A_2)\|\nabla u_2\|_{L^2(\R^d)}^2
+(\gamma -2A_3)\|\nabla u_3\|_{L^2(\R^d)}^2\\
&\ \ \ \ +\left(2\omega -\frac{|\mathbf{c}|^2}{8A_1}\right)\|u_1\|_{L^2(\R^d)}^2
+\left(\omega -\frac{|\mathbf{c}|^2}{8A_2}\right)\|u_2\|_{L^2(\R^d)}^2
+\left(\omega -\frac{|\mathbf{c}|^2}{8A_3}\right)\|u_3\|_{L^2(\R^d)}^2\\
&\ \ \ \ +\sum_{k=1}^d2(I_{1,k}+I_{2,k}+I_{3,k}). 
\end{split}
\]
Because $\omega >\frac{\sigma |\mathbf{c}|^2}{4}$, there exist $A_1,A_2,A_3>0$ such that
\begin{equation}\label{coef_posi}
\begin{split}
    &\alpha -2A_1>0,\ \ \beta -2A_2>0,\ \ \gamma -2A_3>0,\\
    &2\omega -\frac{|\mathbf{c}|^2}{8A_1}>0,\ \ 
    \omega -\frac{|\mathbf{c}|^2}{8A_2}>0,\ \ 
    \omega -\frac{|\mathbf{c}|^2}{8A_3}>0.
\end{split}
\end{equation}
Indeed, if we choose $A_1,A_2,A_3>0$ as
\[
A_1=\frac{1}{4}\left(\alpha +\frac{|\mathbf{c}|^2}{8\omega}\right),\ \ 
A_2=\frac{1}{4}\left(\beta +\frac{|\mathbf{c}|^2}{4\omega}\right),\ \ 
A_3=\frac{1}{4}\left(\gamma +\frac{|\mathbf{c}|^2}{4\omega}\right),
\]
then (\ref{coef_posi}) holds and
we can get
\[
L_{\omega,\mathbf{c}}(U)
\gtrsim (4\omega -\sigma |\mathbf{c}|^2)\|U\|_{\mathcal{H}^1}^2, 
\]
where the implicit constant does not depend on $\mathbf{c}$. 
\end{proof}
Next we prove positivity of $\mu_{\omega,\mathbf{c}}$ and a property of the Nehari functional $K_{\omega,\mathbf{c}}$ under $\omega>\frac{\sigma|\mathbf{c}|^2}{4}$.
\begin{proposition}\label{mu_prop}
Let $d\in \{1,2,3\}$ and $\alpha,\beta,\gamma>0$. 
Assume that $(\omega,\mathbf{c})\in \R\times \R^d$ 
satisfies $\omega >\frac{\sigma |\mathbf{c}|^2}{4}$.
Then, the following properties hold:
\begin{itemize}
    \item[{\rm (i)}] $\mu_{\omega,\mathbf{c}}>0$. 
    \item[{\rm (ii)}] If 
    $U\in \mathcal{H}^1\backslash \{(\mathbf{0},\mathbf{0},\mathbf{0})\}$ satisfies $K_{\omega,\mathbf{c}}(U)<0$\ $($resp\ $\le 0)$, then $L_{\omega,\mathbf{c}}(U)>6\mu_{\omega,\mathbf{c}}$\ $($resp\ $\ge 6\mu_{\omega,\mathbf{c}})$. 
\end{itemize}
\end{proposition}
\begin{proof}
We first prove (i). 
By (\ref{SKL_eq}) and the definition of $\mu_{\omega,\mathbf{c}}$, 
it holds
\begin{equation}\label{mu_L_eq}
\mu_{\omega,\mathbf{c}}=\frac{1}{6}\inf \{L_{\omega,\mathbf{c}}(\Psi)|\ \Psi \in \mathcal{H}^1\backslash \{(\mathbf{0},\mathbf{0},\mathbf{0})\},\ 
K_{\omega,\mathbf{c}}(\Psi)=0\}. 
\end{equation}
Therefore, it suffices to show that there exists $C_1>0$ such that
\[
L_{\omega,\mathbf{c}}(\Psi)\ge C_1
\]
holds
for any $\Psi=(\psi_1,\psi_2,\psi_3)\in \mathcal{H}^1\backslash \{(\mathbf{0},\mathbf{0},\mathbf{0})\}$ with $K_{\omega,\mathbf{c}}(\Psi)=0$. 
If $K_{\omega,\mathbf{c}}(\Psi)=0$, then by the definition of $L_{\omega,\mathbf{c}}$, 
the H\"older inequality, and the Sobolev inequality, 
we have
\[
L_{\omega,\mathbf{c}}(\Psi)=-3N(\Psi)
\lesssim \|\nabla \cdot \psi_3\|_{L^2}\|\psi_1\|_{L^4}\|\psi_2\|_{L^4}
\lesssim \|\Psi\|_{\mathcal{H}^1}^3. 
\]
Furthermore by Proposition~\ref{L_bdd}, we obtain
\[
\|\Psi\|_{\mathcal{H}^1}^2\lesssim L_{\omega,\mathbf{c}}(\Psi)\lesssim \|\Psi\|_{\mathcal{H}^1}^3.
\]
This implies that there exists $\widetilde{C_1}>0$ such that $\|\Psi\|_{\mathcal{H}^1}\ge \widetilde{C_1}$ for any $\Psi \in \mathcal{H}^1\backslash \{(\mathbf{0},\mathbf{0},\mathbf{0})\}$ with $K_{\omega,\mathbf{c}}(\Psi)=0$. 
By setting $C_1:=\widetilde{C_1}^2C(4\omega -\sigma |\mathbf{c}|^2)$, where $C=C_{\omega}>0$ appears in Proposition~\ref{L_bdd}, we get
$L_{\omega,\mathbf{c}}(\Psi)\ge C_1$. 

Next, we prove (ii). 
We assume $U\in \mathcal{H}^1\backslash \{(\mathbf{0},\mathbf{0},\mathbf{0})\}$ and $K_{\omega,\mathbf{c}}(U)<0$. For $\lambda\in \R$, by a simple calculation, we have
\[
K_{\omega,\mathbf{c}}(\lambda U)
=L_{\omega,\mathbf{c}}(\lambda U)+3N(\lambda U)
=\lambda^2L_{\omega,\mathbf{c}}(U)+3\lambda^3N(U).
\]
Therefore, if we set $\lambda$ as
\[
\lambda :=-\frac{L_{\omega,\mathbf{c}}(U)}{3N(U)}, 
\]
then $K_{\omega,\mathbf{c}}(\lambda U)=0$ holds. We can see that $\lambda \in (0,1)$ from the fact that $L_{\omega,\mathbf{c}}(U)+3N(U)=K_{\omega,\mathbf{c}}(U)<0$ and $L_{\omega,\mathbf{c}}(U)>0$. 
As a result, by using (\ref{mu_L_eq}), we get
\[
\mu_{\omega,\mathbf{c}}\le \frac{1}{6}L_{\omega,\mathbf{c}}(\lambda U)
=\frac{\lambda^2}{6}L_{\omega,\mathbf{c}}(U)<\frac{1}{6}L_{\omega,\mathbf{c}}(U).
\]
If we assume $K_{\omega,\mathbf{c}}(U)\le 0$ instead of $K_{\omega,\mathbf{c}}(U)<0$, 
then there is a possibility that $\lambda=1$ holds. 
Therefore, we can only get $\mu_{\omega,\mathbf{c}}\le \frac{1}{6}L_{\omega,\mathbf{c}}(U)$. 
\end{proof}
\begin{proposition}[Existence of a minimizer of $\mathcal{M}_{\omega,\mathbf{c}}$]\label{GS_exist}
Let $d\in \{1,2,3\}$ and $\alpha,\beta,\gamma>0$. 
Assume that $(\omega,\mathbf{c})\in \R\times \R^d$ 
satisfies $\omega >\frac{\sigma |\mathbf{c}|^2}{4}$.
Let $\{U_n\}\subset \mathcal{H}^1$ be 
a sequence satisfying
\[
\lim_{n\rightarrow \infty}S_{\omega,\mathbf{c}}(U_n)=\mu_{\omega,\mathbf{c}},\ \ 
\lim_{n\rightarrow \infty}K_{\omega,\mathbf{c}}(U_n)=0. 
\]
Then, there exist $\{y_n\}\subset \R^d$  such that 
the sequence $\{U_n(\cdot -y_n)\}$ 
has a subsequence which converges to some $V\in \mathcal{H}^1\backslash \{(\mathbf{0},\mathbf{0},\mathbf{0})\}$
strongly in $\mathcal{H}^1$. 
Furthermore it holds that $V\in \mathcal{M}_{\omega,\mathbf{c}}$. Namely $\mathcal{M}_{\omega,\mathbf{c}}$ is not empty.

In addition, if $\{U_n\}$ satisfies 
\begin{equation}\label{ex_gs_add_cond}
\limsup_{n\rightarrow \infty}G(U_n)\ge \eta
\end{equation}
for some $\eta >0$, where $G$ is defined by (\ref{defG}), then it holds that $V\in \mathcal{M}_{\omega,\mathbf{c}}^*(\eta)$. 
\end{proposition}
We remark that the latter statement in this proposition will be applied to prove Theorem \ref{stability} in Subsection \ref{proof of stability}.

%
%
To prove Proposition~\ref{GS_exist}, 
we use the following Lieb's compactness theorem.
\begin{lemma}[Lieb's compactness theorem \cite{Lieb83}]\label{Lieb}
Let $\{F_n\}$ be a bounded sequence 
in $\mathcal{H}^1$. 
Assume that there exist $p\in (0,2^*)$ such that 
$\limsup_{n\rightarrow \infty}\|F_n\|_{L^p}>0$. 
Then there exists $\{y_n\}\subset \R^d$ and 
$F\in \mathcal{H}^1\backslash \{(\mathbf{0},\mathbf{0},\mathbf{0})\}$ 
such that $\{F_n(\cdot -y_n)\}$ 
has a subsequence that converges to F weakly in 
$\mathcal{H}^1$, where 
$p^*=\infty$ if $d=1$ or $2$, $p^*=\frac{2d}{d-2}$ if $d\ge 3$. 
\end{lemma}
\begin{lemma}[Special case of Brezis-Lieb's lemma \cite{BL83}]\label{Brezis_Lieb}
Assume that $\{f_n\}$ be a bounded sequence in $L^2(\R^d)$ 
and $f_n\rightharpoonup f$ weakly in $L^2(\R^d)$. 
Then it holds that
\[
\|f_n\|_{L^2}^2-\|f_n-f\|_{L^2}^2-\|f\|_{L^2}^2\ \rightarrow\ 0.
\]
\end{lemma}
%
%
\begin{remark}
Lemma~\ref{Brezis_Lieb} also holds 
if we replace $L^2(\R^d)$ by $(L^2(\R^d))^d$. 
\end{remark}
\begin{proof}[Proof of Proposition~\ref{GS_exist}]
We divide the proof into 6 steps as follows.\\
{\bf Step\ 1}.\ We prove the boundedness 
of $\{U_n\}$ in $\mathcal{H}^1$. 
By (\ref{SKL_eq}), we have
\begin{equation}\label{L_limit}
L_{\omega,\mathbf{c}}(U_n)=6S_{\omega,\mathbf{c}}(U_n)-2K_{\omega,\mathbf{c}}(U_n)
\ \rightarrow\ 6\mu_{\omega,\mathbf{c}}. 
\end{equation}
Thus $\{L_{\omega,\mathbf{c}}(U_n)\}$ is a bounded sequence. 
This and Proposition~\ref{L_bdd} imply that $\{U_n\}$ is bounded in $\mathcal{H}^1$. \\

\noindent {\bf Step\ 2}. We prove $\limsup_{n\rightarrow \infty}\|U_n\|_{L^4}>0$. 
We write $U_n=(u_{1n},u_{2n},u_{3n})$\ ($u_{jn}\in (H^1(\R^d))^d$).
By Step\ 1, there exists $C>0$ such that $\|U_n\|_{\mathcal{H}^1}\le C$ holds for any $n\in \N$. 
Therefore, by the H\"older inequality, we have
\[
N(U_n)=\left|{\rm Re}(u_{3n},\nabla (u_{1n}\cdot \overline{u_{2n}})_{L^2}\right|
\le \|\nabla \cdot u_{3n}\|_{L^2}\|u_{1n}\|_{L^4}\|u_{2n}\|_{L^4}
\le C\|U_n\|_{L^4}^2. 
\]
Assume $\lim_{n\rightarrow \infty}\|U_n\|_{L^4}=0$. Then we obtain $\lim_{n\rightarrow \infty}N(U_n)=0$. 
This implies that
\[
L_{\omega,\mathbf{c}}(U_n)=K_{\omega,\mathbf{c}}(U_n)-3N(U_n)\ \rightarrow\ 0. 
\]
This contradicts to (\ref{L_limit}) 
because $\mu_{\omega,\mathbf{c}}>0$ by Proposition~\ref{mu_prop}\ (i). 
Therefore we get $\lim_{n\rightarrow \infty}\|U_n\|_{L^4}\ne 0$. 
Namely, 
\[
\limsup_{n\rightarrow \infty}\|U_n\|_{L^4}>0. 
\]

\noindent {\bf Step\ 3}. In this step, we apply Lemma~\ref{Lieb}. 
By Step\ 1 and Step\ 2, $\{U_n\}$ satisfies 
the assumptions in Lemma~\ref{Lieb} for $p=4$. 
Therefore, there exist $\{y_n\}\subset \R^d$ 
and $V\in \mathcal{H}^1\backslash \{(\mathbf{0},\mathbf{0},\mathbf{0})\}$ such that 
$\{U_n(\cdot -y_n)\}$ has a subsequence $\{V_{n}\}$ that converges to $V$ 
weakly in $\mathcal{H}^1$. 
We set
\[
\begin{split}
&V_n=(v_{1n},v_{2n},v_{3n}):=U_n(\cdot -y_n)\ \ (v_{jn}\in (H^1(\R^d))^d),\\
&V=(v_{1},v_{2},v_{3})\ \ (v_{j}\in (H^1(\R^d))^d). 
\end{split}
\]
We can assume $V_n\ne V$ for any $n\in \N$ without loss of generality. \\

\noindent {\bf Step\ 4}. We prove
\begin{equation}\label{LV_limit}
L_{\omega,\mathbf{c}}(V_n)-L_{\omega,\mathbf{c}}(V_n-V)-L_{\omega,\mathbf{c}}(V)\ \rightarrow\ 0
\end{equation}
and
\begin{equation}\label{KV_limit}
K_{\omega,\mathbf{c}}(V_n)-K_{\omega,\mathbf{c}}(V_n-V)-K_{\omega,\mathbf{c}}(V)\ \rightarrow\ 0. 
\end{equation}
First, we show (\ref{LV_limit}). 
As in the proof of Proposition~\ref{L_bdd}, we can write
\begin{equation}\label{L_fomula}
\begin{split}
L_{\omega,\mathbf{c}}(V)
&=\sum_{j=1}^3\left(C_{1,j}\|\nabla v_j\|_{L^2}^2+C_{2,j}\|v_j\|_{L^2}^2+
\sum_{k=1}^dC_{3,j,k}\|\partial_kv_j+c_{j,k}v_j\|_{L^2}^2\right)\\
&=\sum_{j=1}^3\left(C_{1,j}\sum_{k=1}^d\|\partial_k v_j\|_{L^2}^2+C_{2,j}\|v_j\|_{L^2}^2+
\sum_{k=1}^dC_{3,j,k}\|\partial_kv_j+c_{j,k}v_j\|_{L^2}^2\right)
\end{split}
\end{equation}
for some positive constants $C_{1,j}$, $C_{2,j}$, $C_{3,j,k}$ and 
constants $c_{j,k}\in \C$. 
Because $V_n\rightharpoonup V$ weakly in $\mathcal{H}^1$, 
the sequences
$\{v_{jn}\}$ and $\{\partial_kv_{jn}\}$ are bounded in $(L^2(\R^d))^d$,  
and it holds that
\begin{equation}\label{v_weak_l2}
\partial_k v_{jn}\rightharpoonup \partial_k v_j,\ \ \ \ 
v_{jn}\rightarrow v_j\ \ \ \ {\rm weakly}\ {\rm in}\ L^2(\R^d)^d
\end{equation}
by taking subsequences. 
Therefore, by applying Lemma~\ref{Brezis_Lieb}, 
we get (\ref{LV_limit}). 

Next, we show (\ref{KV_limit}). 
We can write
\begin{equation}\label{N_rep_norm}
2N(V)=2{\rm Re}(v_3,\nabla (v_1\cdot v_2))_{L^2}
=\|\nabla \cdot v_3-v_1\cdot \overline{v_2}\|_{L^2}^2
-\|\nabla \cdot v_3\|_{L^2}^2-\|v_1\cdot \overline{v_2}\|_{L^2}^2. 
\end{equation}
Because $V_n\rightarrow V$ weakly in $\mathcal{H}^1$, 
the sequences $\{v_{1n}\}$, $\{v_{2n}\}$, and $\{v_{3n}\}$ are
bounded in $(H^1(\R^d))^d$. 
This implies that the sequences $\{\nabla \cdot v_{3n}\}$ and $\{v_{1n}\cdot \overline{v_{2n}}\}$ 
are bounded in $L^2(\R^d)$ because
\[
\|\nabla \cdot v_{3n}\|_{L^2}\le \|v_{3n}\|_{H^1}
\]
and 
\[
\|v_{1n}\cdot \overline{v_{2n}}\|_{L^2}
\le \|v_{1n}\|_{L^4}\|v_{2n}\|_{L^4}
\le \|v_{1n}\|_{H^1}\|v_{2n}\|_{H^1} 
\]
hold by the H\"older inequality and the Sobolev inequality. 
Therefore, by taking subsequences, we obtain 
\[
\nabla \cdot v_{3n}\rightharpoonup \nabla \cdot v_{3},\ \ \ \ 
v_{1n}\cdot \overline{v_{2n}}\rightharpoonup v_1\cdot \overline{v_2}\ \ \ \ {\rm weakly}\ {\rm in}\ L^2(\R^d). 
\]
Furthermore, by (\ref{N_rep_norm}) and applying Lemma~\ref{Brezis_Lieb}, 
we have
\begin{equation}\label{NV_limit}
N(V_n)-N(V_n-V)-N(V)\ \rightarrow\ 0. 
\end{equation}
Because $K_{\omega,\mathbf{c}}(U)=L_{\omega,\mathbf{c}}(U)+3N(U)$, we get (\ref{KV_limit}) by (\ref{LV_limit}) and (\ref{NV_limit}). \\

\noindent {\bf Step\ 5}. In this step, we prove $K_{\omega,\mathbf{c}}(V)\le 0$ by contradiction. 
We assume $K_{\omega,\mathbf{c}}(V)>0$. (Then, $V\ne 0$ by the definition of $K_{\omega,\mathbf{c}}$.)
Because
\[
K_{\omega,\mathbf{c}}(V_n)=K_{\omega,\mathbf{c}}(U_n)\rightarrow 0, 
\]
By (\ref{KV_limit}), we have
\[
\begin{split}
K_{\omega,\mathbf{c}}(V_n-V)
&=-K_{\omega,\mathbf{c}}(V)+K_{\omega,\mathbf{c}}(V_n)
-(K_{\omega,\mathbf{c}}(V_n)-K_{\omega,\mathbf{c}}(V_n-V)-K_{\omega,\mathbf{c}}(V))\\
&\rightarrow -K_{\omega,\mathbf{c}}(V)<0.
\end{split}
\]
Therefore, if $n\in \N$ is large enough, then it holds $K_{\omega,\mathbf{c}}(V_n-V)<0$. 
Since $V_n-V\ne 0$, 
this and Proposition~\ref{mu_prop}\ (ii) imply that 
$L_{\omega,\mathbf{c}}(V_n-V)>6\mu_{\omega,\mathbf{c}}$ for sufficiently large $n\in \N$. 
Because by (\ref{L_limit}),
\begin{equation}\label{LV_mu}
L_{\omega,\mathbf{c}}(V_n)=L_{\omega,\mathbf{c}}(U_n)\rightarrow 6\mu_{\omega,\mathbf{c}}
\end{equation}
holds and we obtain 
\[
\begin{split}
L_{\omega,\mathbf{c}}(V)
&=L_{\omega,\mathbf{c}}(V_n)-L_{\omega,\mathbf{c}}(V_n-V)
-(L_{\omega,\mathbf{c}}(V_n)-L_{\omega,\mathbf{c}}(V_n-V)-L_{\omega,\mathbf{c}}(V))\\
&<L_{\omega,\mathbf{c}}(V_n)-6\mu_{\omega,\mathbf{c}}
-(L_{\omega,\mathbf{c}}(V_n)-L_{\omega,\mathbf{c}}(V_n-V)-L_{\omega,\mathbf{c}}(V))\\
&\rightarrow 0
\end{split}
\]
by (\ref{LV_limit}). 
But this leads to a contradiction because by Proposition~\ref{L_bdd} and $V\ne 0$, 
\[
L_{\omega,\mathbf{c}}(V)\gtrsim \|V\|_{\mathcal{H}^1}^2>0
\]
holds.\\

\noindent {\bf Step\ 6}. Finally. we prove 
$V_n\rightarrow V$ strongly in $\mathcal{H}^1$ and $V\in \mathcal{M}_{\omega,\mathbf{c}}$. 
By Step\ 5 and Proposition~\ref{mu_prop}\ (ii), 
we have $L_{\omega, \mathbf{c}}(V)\ge 6\mu_{\omega,\mathbf{c}}$. 
On the other hand, by (\ref{L_fomula}), we obtain
\[
\begin{split}
L_{\omega,\mathbf{c}}(V)
&\le \sum_{j=1}^3\left(C_{1,j}\sum_{k=1}^d
\liminf_{n\rightarrow \infty}\|\partial_k v_{jn}\|_{L^2}^2+C_{2,j}
\liminf_{n\rightarrow \infty}\|v_{jn}\|_{L^2}^2+
\sum_{k=1}^dC_{3,j,k}
\liminf_{n\rightarrow \infty}\|\partial_kv_{jn}+c_{j,k}v_{jn}\|_{L^2}^2\right)\\
&\le \lim_{n\rightarrow \infty}
\sum_{j=1}^3\left(C_{1,j}\sum_{k=1}^d
\|\partial_k v_{jn}\|_{L^2}^2+C_{2,j}
\|v_{jn}\|_{L^2}^2+
\sum_{k=1}^dC_{3,j,k}
\|\partial_kv_{jn}+c_{j,k}v_{jn}\|_{L^2}^2\right)\\
&= \lim_{n\rightarrow \infty}L_{\omega,\mathbf{c}}(V_n),
\end{split}
\]
where the first inequality is 
obtained by (\ref{v_weak_l2}). 
Therefore, we get $L_{\omega,\mathbf{c}}(V)\le 6\mu_{\omega,\mathbf{c}}$ by (\ref{LV_mu}). 
As a result, we have $L_{\omega,\mathbf{c}}(V)= 6\mu_{\omega,\mathbf{c}}$. 
Furthermore, by Proposition~\ref{L_bdd}, (\ref{LV_limit}), and (\ref{LV_mu}), we obtain
\[
\begin{split}
\|V_n-V\|_{\mathcal{H}^1}^2
&\lesssim L_{\omega,\mathbf{c}}(V_n-V)\\
&=L_{\omega,\mathbf{c}}(V_n)-L_{\omega,\mathbf{c}}(V)
-(L_{\omega,\mathbf{c}}(V_n)-L_{\omega,\mathbf{c}}(V_n-V)-L_{\omega,\mathbf{c}}(V))\\
&\rightarrow 0.
\end{split}
\]
Namely $V_n\rightarrow V$ strongly 
in $\mathcal{H}^1$. 
Because $L_{\omega,\mathbf{c}}(V)= 6\mu_{\omega,\mathbf{c}}$ 
(In particular $L_{\omega,\mathbf{c}}(V)\le 6\mu_{\omega,\mathbf{c}}$), 
we have $K_{\omega,\mathbf{c}}(V)\ge 0$ by Proposition~\ref{mu_prop}\ (ii). 
This and Step 5 imply $K_{\omega,\mathbf{c}}(V)=0$. 
Thus, we obtain
\[
S_{\omega,\mathbf{c}}(V)
=\frac{1}{3}K_{\omega,\mathbf{c}}(V)+\frac{1}{6}L_{\omega,\mathbf{c}}(V)=\mu_{\omega,\mathbf{c}}
\]
by (\ref{SKL_eq}) and it implies $V\in \mathcal{M}_{\omega,\mathbf{c}}$. 

Now, we assume that the additional condition (\ref{ex_gs_add_cond}) holds. 
Namely, we assume
\[
\limsup_{n\rightarrow \infty}G(U_n)=
\limsup_{n\rightarrow \infty}\left\{(4-2d)\omega Q(U_n)+(3-d)\mathbf{c}\cdot \mathbf{P}(U_n)\right\}\ge \eta.
\]
By the definitions of $Q$ and $\mathbf{P}$,
\[
\begin{split}
|G(V_n)-G(V)|
&\lesssim \|V_n-V\|_{L^2}^2
+\|V_n-V\|_{L^2}\|\nabla (V_n-V)\|_{L^2}
+\|V\|_{L^2}\|\nabla (V_n-V)\|_{L^2}
+\|V_n-V\|_{L^2}\|\nabla V\|_{L^2}\\
&\lesssim \|V_n-V\|_{\mathcal{H}^1}^2+\|V_n-V\|_{\mathcal{H}^1}
\ \rightarrow\ 0
\end{split}
\]
holds. We also have
\[
G(V)=G(V_n)-(G(V_n)-G(V))=G(U_n)-(G(V_n)-G(V)). 
\]
These imply $G(V)=\limsup_{n\rightarrow \infty}G(U_n)\ge \eta$ and we get $V\in \mathcal{M}_{\omega,\mathbf{c}}^*(\eta)$.
\end{proof}
\begin{proposition}\label{MGeq}
Let $d\in \{1,2,3\}$ and $\alpha,\beta,\gamma>0$. 
Assume that $(\omega,\mathbf{c})\in \R\times \R^d$ 
satisfies $\omega >\frac{\sigma |\mathbf{c}|^2}{4}$.
Then, we have $\mathcal{M}_{\omega,\mathbf{c}}=\mathcal{G}_{\omega,\mathbf{c}}$. 
\end{proposition}
\begin{proof}
We first prove $\mathcal{M}_{\omega,\mathbf{c}}\subset \mathcal{G}_{\omega,\mathbf{c}}$. 
Let $\Psi =(\psi_1,\psi_2,\psi_3)\in \mathcal{M}_{\omega,\mathbf{c}}$. 
Since $\Psi$ is a minimizer of $S_{\omega,\mathbf{c}}(\Phi)$ under the condition $K_{\omega,\mathbf{c}}(\Phi)=0$, 
there exists a Lagrange multiplier $\eta \in \R$ such that
\[
D_jS_{\omega,\mathbf{c}}(\Psi)=\eta D_jK_{\omega,\mathbf{c}}(\Psi),\ \ j=1,2,3.
\]
Therefore, we have
\[
0=K_{\omega,\mathbf{c}}(\Psi)
=\partial_{\lambda}S_{\omega,\mathbf{c}}(\lambda \Psi)\big|_{\lambda =1}
=\sum_{j=1}^3\langle D_jS_{\omega,\mathbf{c}}(\Psi),\psi_j\rangle
=\eta \sum_{j=1}^3\langle D_jK_{\omega,\mathbf{c}}(\Psi),\psi_j\rangle
=\eta \partial_{\lambda}K_{\omega,\mathbf{c}}(\lambda \Psi)\big|_{\lambda =1}. 
\]
We also note that the following identities hold:
\[
\begin{split}
\partial_{\lambda}K_{\omega,\mathbf{c}}(\lambda \Psi)\big|_{\lambda =1}
&=\partial_{\lambda}(2L(\lambda\Psi)+3N(\lambda\Psi)+2\omega Q(\lambda\Psi)+2\mathbf{c}\cdot \mathbf{P}(\lambda\Psi))\big|_{\lambda=1}\\
&=4L(\Psi)+9N(\Psi)+4\omega Q(\Psi)+4\mathbf{c}\cdot \mathbf{P}(\Psi)\\
&=3K_{\omega,\mathbf{c}}(\Psi)-L_{\omega,\mathbf{c}}(\Psi)=-L_{\omega,\mathbf{c}}(\Psi). 
\end{split}
\]
Because $\Psi \ne 0$, we obtain $\partial_{\lambda}K_{\omega,\mathbf{c}}(\lambda \Psi)\big|_{\lambda =1}<0$ 
by Proposition~\ref{L_bdd}. 
This implies $\eta =0$ and we get 
\[
D_jS_{\omega,\mathbf{c}}(\Psi)=0,\ \  j=1,2,3.
\]
This means $\Psi \in \mathcal{E}_{\omega,\mathbf{c}}$. 
Furthermore, 
for any $\Theta \in \mathcal{E}_{\omega,\mathbf{c}}$, we have $K_{\omega,\mathbf{c}}(\Theta)=0$ 
by Remark~\ref{EK_rem}. 
Thus, by the definition of $\mu_{\omega,\mathbf{c}}$, 
we obtain
\[
S_{\omega,\mathbf{c}}(\Psi)=\mu_{\omega,\mathbf{c}}\le S_{\omega,\mathbf{c}}(\Theta ). 
\]
Therefore, we get $\Psi \in \mathcal{G}_{\omega,\mathbf{c}}$. 

Next, we prove $\mathcal{G}_{\omega,\mathbf{c}}\subset \mathcal{M}_{\omega,\mathbf{c}}$. 
Let $\Psi \in \mathcal{G}_{\omega,\mathbf{c}}$. 
By Proposotion~\ref{GS_exist}, 
there exists $\Phi \in \mathcal{M}_{\omega,\mathbf{c}}$. 
Because $\mathcal{M}_{\omega,\mathbf{c}}\subset \mathcal{G}_{\omega,\mathbf{c}}\subset \mathcal{E}_{\omega,\mathbf{c}}$, 
we have $\Phi \in \mathcal{E}_{\omega,\mathbf{c}}$. 
This implies that 
\[
S_{\omega,\mathbf{c}}(\Psi)\le S_{\omega,\mathbf{c}}(\Phi)=\mu_{\omega,\mathbf{c}}. 
\]
On the other hand, by Remark~\ref{EK_rem}, we have $K_{\omega,\mathbf{c}}(\Psi)=0$ 
since $\Psi \in \mathcal{G}_{\omega,\mathbf{c}}\subset \mathcal{E}_{\omega,\mathbf{c}}$. 
Therefore, by the definition of $\mu_{\omega,\mathbf{c}}$, we get
\[
\mu_{\omega,\mathbf{c}}\le S_{\omega,\mathbf{c}}(\Psi). 
\]
As a result, we obtain $S_{\omega,\mathbf{c}}(\Psi)=\mu_{\omega,\mathbf{c}}$ and 
this implies $\Psi \in \mathcal{M}_{\omega,\mathbf{c}}$. 
\end{proof}

\begin{proof}[Proof of Theorem \ref{GS_exist_th}]
Theorem \ref{GS_exist_th} can be proved by combining Proposition \ref{GS_exist} and Proposition \ref{MGeq}.
\end{proof}

Next we prepare fundamental identities about the ground-state energy level $\mu_{\omega,\mathbf{c}}$, which will be applied to prove the stability results in Section~\ref{GWP_Sta}.  
\begin{proposition}\label{mu_scale}
Let $d\in \{1,2,3\}$ and $\alpha,\beta,\gamma>0$. 
Assume that $(\omega,\mathbf{c})\in \R\times \R^d$ 
satisfies $\omega >\frac{\sigma |\mathbf{c}|^2}{4}$.
Then, we have 
\[
\mu_{\omega,\mathbf{c}}=\omega^{2-\frac{d}{2}}\mu_{1,\frac{\mathbf{c}}{\sqrt{\omega}}}.
\]
\end{proposition}
\begin{proof}
By Propositions \ref{GS_exist} and \ref{MGeq}, there exists $\Psi \in \mathcal{M}_{1,\frac{\mathbf{c}}{\sqrt{\omega}}}=\mathcal{G}_{1,\frac{\mathbf{c}}{\sqrt{\omega}}}$, which implies that $\Psi \in \mathcal{E}_{1,\frac{\mathbf{c}}{\sqrt{\omega}}}$.
Furthermore, we set
\[
\Psi_{\omega}(x):=\omega^{\frac{1}{2}}\Psi(\omega^{\frac{1}{2}}x). 
\]
Then, we have $\Psi_{\omega}\in \mathcal{E}_{\omega,\mathbf{c}}$ and by Remark~\ref{sol_scale},
\begin{equation}\label{psi_scale_eq}
S_{\omega,\mathbf{c}}(\Psi_{\omega})=\omega^{2-\frac{d}{2}}S_{1,\frac{\mathbf{c}}{\sqrt{\omega}}}(\Psi).
\end{equation}
Let $\Theta \in \mathcal{E}_{\omega,\mathbf{c}}$. 
We put 
\[
\Theta^{\omega}(x):=\omega^{-\frac{1}{2}}\Theta (\omega^{-\frac{1}{2}}x). 
\]
Then, we can see that $\Theta^{\omega}\in \mathcal{E}_{1,\frac{\mathbf{c}}{\sqrt{\omega}}}$ and 
\begin{equation}\label{theta_scale_eq}
S_{1,\frac{\mathbf{c}}{\sqrt{\omega}}}(\Theta^{\omega})=\omega^{\frac{d}{2}-2}S_{\omega,\mathbf{c}}(\Theta).
\end{equation}
On the other hand, because $\Psi$ is a minimizer of $\mathcal{M}_{1,\frac{\mathbf{c}}{\sqrt{\omega}}}$, we obtain
\[
S_{1,\frac{\mathbf{c}}{\sqrt{\omega}}}(\Psi)\le S_{1,\frac{\mathbf{c}}{\sqrt{\omega}}}(\Theta^{\omega})
\]
This implies that by (\ref{psi_scale_eq}) and (\ref{theta_scale_eq})
$S_{\omega,\mathbf{c}}(\Psi_{\omega})\le S_{\omega,\mathbf{c}}(\Theta)$. This implies $\Psi_{\omega}\in \mathcal{G}_{\omega,\mathbf{c}}=\mathcal{M}_{\omega,\mathbf{c}}$. 
As a result, we get
\[
\mu_{\omega,\mathbf{c}}=S_{\omega,\mathbf{c}}(\Psi_{\omega})
=\omega^{2-\frac{d}{2}}S_{1,\frac{\mathbf{c}}{\sqrt{\omega}}}(\Psi)
=\omega^{2-\frac{d}{2}}\mu_{1,\frac{\mathbf{c}}{\sqrt{\omega}}}. 
\]
\end{proof}
\begin{proposition}
Let $d\in \{1,2,3\}$ and $\alpha,\beta,\gamma>0$. 
Assume that $(\omega,\mathbf{c})\in \R\times \R^d$ 
satisfies $\omega >\frac{\sigma |\mathbf{c}|^2}{4}$.
If $\Phi\in \mathcal{M}_{\omega,\mathbf{c}}$, then 
it holds that
\begin{equation}\label{M_scale_eq}
2\omega Q(\Phi)
+\mathbf{c}\cdot \mathbf{P}(\Phi)
=(4-d)\omega^{2-\frac{d}{2}}
\mu_{1,\frac{\mathbf{c}}{\sqrt{\omega}}}. 
\end{equation}
\end{proposition}
\begin{proof}
Because $\Phi$ is a minimizer of 
$\mu_{\omega,\mathbf{c}}$, it holds 
\begin{equation}\label{LM_eq}
K_{\omega, \mathbf{c}}(\Phi)
=2L(\Phi)+3N(\Phi)+2\omega Q(\Phi)+2\mathbf{c}\cdot \mathbf{P}(\Phi)=0.
\end{equation}
We also note that $\Phi$ is a weak solution to (\ref{ellip}) by Proposition~\ref{MGeq}. 
Therefore, by Proposition~\ref{LN_iden}, the identity
\begin{equation}\label{LN_eq}
2L(\Phi)+\left(\frac{d}{2}+1\right)N(\Phi)+\mathbf{c}\cdot \mathbf{P}(\Phi)=0
\end{equation}
holds. The identities (\ref{LM_eq}) and (\ref{LN_eq}) imply
\[
L(\Phi)=\frac{d+2}{4-d}\omega Q(\Phi)+\frac{d-1}{4-d}\mathbf{c}\cdot \mathbf{P}(\Phi),\ \ \ \text{and}\ \ \ 
N(\Phi)=-\frac{2}{4-d}\left(2\omega Q(\Phi)
+\mathbf{c}\cdot \mathbf{P}(\Phi)\right). 
\]
Hence we obtain
\[
\begin{split}
\mu_{\omega,\mathbf{c}}
=S_{\omega,\mathbf{c}}(\Phi)
&=L(\Phi)+N(\Phi)+\omega Q(\Phi)+\mathbf{c}\cdot \mathbf{P}(\Phi)
=\frac{1}{4-d}\left(2\omega Q(\Phi)+\mathbf{c}\cdot \mathbf{P}(\Phi)\right). 
\end{split}
\]
On the other hand, by Proposition~\ref{mu_scale}, it holds that
$\mu_{\omega,\mathbf{c}}=\omega^{2-\frac{d}{2}}\mu_{1,\frac{\mathbf{c}}{\sqrt{\omega}}}$. Therefore, we get (\ref{M_scale_eq}). 
\end{proof}
\section{Proof of global well-posedness 
and stability results}\label{GWP_Sta}
In this section, we give the proofs of 
Theorems~\ref{gwp}, ~\ref{stability} and Corollaries~\ref{gwp_2d}, ~\ref{stab_1d_result}. 
\subsection{Proof of global well-posedness}
We define the subsets in the energy space $\mathcal{H}^1$
\begin{equation}\label{A_pm_def}
\begin{split}
   \mathcal{A}^{+}_{\omega,\mathbf{c}}&:=\left\{\Psi\in \mathcal{H}^1\backslash \{(\mathbf{0},\mathbf{0},\mathbf{0})\}\ |\ S_{\omega,\mathbf{c}}(\Psi)< \mu_{\omega,\mathbf{c}},\ K_{\omega,\mathbf{c}}(\Psi)> 0\right\},\\
   \mathcal{A}^{-}_{\omega,\mathbf{c}}&:=\left\{\Psi\in \mathcal{H}^1\backslash \{(\mathbf{0},\mathbf{0},\mathbf{0})\}\ |\ S_{\omega,\mathbf{c}}(\Psi)< \mu_{\omega,\mathbf{c}},\ K_{\omega,\mathbf{c}}(\Psi)< 0\right\}.
 \end{split}
\end{equation}
We prove that $\mathcal{A}^{\pm}_{\omega,\mathbf{c}}$ is invariant under the flow of (\ref{NLS}).
\begin{proposition}\label{A_sol_inv}
Let $d\in \{1,2,3\}$ and $\alpha,\beta,\gamma>0$. 
Assume that $(\omega,\mathbf{c})\in \R\times \R^d$ 
satisfies $\omega >\frac{\sigma |\mathbf{c}|^2}{4}$.
Then the sets $\mathcal{A}^{\pm}_{\omega,\mathbf{c}}$ are 
invariant under the flow of {\rm (\ref{NLS})}. 
More precisely, 
if $U_0$ belongs to $\mathcal{A}^{+}_{\omega,\mathbf{c}}$ $($resp $\mathcal{A}^{-}_{\omega,\mathbf{c}})$, 
then the solution $U(t)$ of {\rm (\ref{NLS})} with $U(0)=U_0$ also belongs to 
$\mathcal{A}^{+}_{\omega,\mathbf{c}}$ $($resp $\mathcal{A}^{-}_{\omega,\mathbf{c}})$ 
for all $t\in I_{\max}$. 
\end{proposition}
\begin{proof}
Let $U_0\in \mathcal{A}^{+}_{\omega,\mathbf{c}}$ 
and $U$ be the solution to {\rm (\ref{NLS})} with $U(0)=U_0$ on $I_{\max}$. 
Because $S_{\omega,\mathbf{c}}(U_0)<\mu_{\omega,\mathbf{c}}$ holds and 
$S_{\omega,\mathbf{c}}$ is a conserved quantity, 
we have $S_{\omega,\mathbf{c}}(U(t))<\mu_{\omega,\mathbf{c}}$ for all $t\in I_{\max}$.
Furthermore, because $U_0\ne 0$ and $Q$ is a conserved quantity, 
we obtain $Q(U(t))=Q(U_0)\ne 0$. It implies that $U(t)\ne 0$ for all $t\in I_{\max}$. 
Therefore, it suffices to show that $K_{\omega,\mathbf{c}}(U(t))>0$ for all $t\in I_{\max}$. 
Assume $K_{\omega,\mathbf{c}}(U(t_*))<0$ for some $t_*\in I_{\max}$. 
Then, by the continuity of $t\mapsto K_{\omega,\mathbf{c}}(U(t))$, 
there exists $t_0\in (0,t_*)$ such that $K_{\omega,\mathbf{c}}(U(t_0))=0$. 
This implies $U(t_0)=0$ because $S_{\omega,\mathbf{c}}(U(t_0))<\mu_{\omega,\mathbf{c}}$. 
It is a contradiction. 
The invariance of $\mathcal{A}^{-}_{\omega,\mathbf{c}}$ can be proved 
by the same manner. 
\end{proof}
Now we prove our global well-posedness results (Theorem \ref{gwp} and Corollary~\ref{gwp_2d}). 
\begin{proof}[Proof of Theorem~\ref{gwp}]
Let $U$ be a solution on the maximal existence time interval $I_{\max}$ to (\ref{NLS}) with $U(0)=U_0\in \mathcal{H}^1\backslash\{0\}$ satisfying
\[
S_{\omega,\mathbf{c}}(U_0)<\mu_{\omega,\mathbf{c}},\ \ 
K_{\omega,\mathbf{c}}(U_0)>0.  
\]
That is $U_0\in \mathcal{A}^{+}_{\omega,\mathbf{c}}$. 
By Proposition~\ref{A_sol_inv}, 
$U(t)\in \mathcal{A}^{+}_{\omega,\mathbf{c}}$ 
holds for any $t\in I_{\max}$. 
This implies $K_{\omega,\mathbf{c}}(U(t))>0$ 
and by (\ref{SKL_eq}) we have
\[
L_{\omega,\mathbf{c}}(U(t))\le 6S_{\omega,\mathbf{c}}(U(t)).
\]
Thus by Proposition~\ref{L_bdd}, we obtain
\[
\|U(t)\|_{\mathcal{H}^1}^2
\le \frac{6}{C}S_{\omega,\mathbf{c}}(U(t))
=\frac{6}{C}S_{\omega,\mathbf{c}}(U_0) 
\]
for any $t\in I_{\max}$ and some $C>0$ independent of $t$,
since $S_{\omega,\mathbf{c}}$ is a conserved quantity, which implies that $|I_{\max}|=\infty$ by the blow-up alternative.
\end{proof}
\begin{proof}[Proof of Corollary~\ref{gwp_2d}]
We assume $\Phi \in \mathcal{M}_{1,\mathbf{c}_0}$ 
with $|\mathbf{c}_0|<\frac{2}{\sqrt{\sigma}}$ 
and set $\mathbf{c}:=\sqrt{\omega}\ \mathbf{c}_0$ 
for $\omega >0$. 
Furthermore, we can assume $Q(U_0)>0$ because 
$U_0=0$ when $Q(U_0)=0$. 
Thanks to Theorem~\ref{gwp}, it is enough to show that
\begin{equation}\label{SK_est_gwp}
S_{\omega,\mathbf{c}}(U_0)<\mu_{\omega,\mathbf{c}},\ \ 
K_{\omega,\mathbf{c}}(U_0)>0
\end{equation}
hold for some $\omega >0$. 
Because
\[
\begin{split}
K_{\omega,\mathbf{c}}(U_0)
&=2L(U_0)+3N(U_0)+2\omega Q(U_0)+\mathbf{c}\cdot \mathbf{P}(U_0)\\
&=2\omega Q(U_0)+\sqrt{\omega}\ \mathbf{c}_0\cdot \mathbf{P}(U_0)
+2L(U_0)+3N(U_0)
\end{split}
\]
and $Q(U_0)>0$, 
the second inequality in (\ref{SK_est_gwp}) holds 
for sufficiently large $\omega >0$. 

Next, to prove the first inequality in (\ref{SK_est_gwp}), 
we define the function $f:(0,\infty)\rightarrow \R$ as
\[
f(\omega):=\mu_{\omega,\mathbf{c}}-S_{\omega,\mathbf{c}}(U_0)
=\mu_{\omega,\sqrt{\omega}\ \mathbf{c}_0}-S_{\omega,\sqrt{\omega}\ \mathbf{c}_0}(U_0).
\]
Because $\Phi \in \mathcal{M}_{1,\mathbf{c}_0}$, we have
\[
\mathbf{c}_0\cdot \mathbf{P}(\Phi)=-2(L(\Phi)+N(\Phi))=-2E(\Phi) 
\]
by Proposition~\ref{LN_iden} with $d=2$. 
This identity and Proposition~\ref{mu_scale} with $d=2$ imply that
\[
\mu_{\omega,\sqrt{\omega}\ \mathbf{c}_0}
=\omega \mu_{1,\mathbf{c}_0}
=\omega S_{1,\mathbf{c}_0}(\Phi)
=\omega (Q(\Phi)-E(\Phi)). 
\]
Therefore, we obtain
\[
\begin{split}
f(\omega)
&=\omega (Q(\Phi)-E(\Phi))
-(E(U_0)+\omega Q(U_0)+\sqrt{\omega}\ \mathbf{c}_0\cdot \mathbf{P}(U_0))\\
&=\omega (Q(\Phi)-E(\Phi)-Q(U_0))-\sqrt{\omega}\ \mathbf{c}_0\cdot \mathbf{P}(U_0)-E(U_0). 
\end{split}
\]
This says that if $Q(U_0)<Q(\Phi)-E(\Phi)$, 
then $f(\omega)>0$ holds for sufficiently large $\omega >0$. 
\end{proof}
We define the subsets in the energy space $\mathcal{H}^1$
\begin{equation}\label{gwp_set_B}
\begin{split}
   \mathcal{B}^{+}_{\omega,\mathbf{c}}&:=\left\{\Psi\in \mathcal{H}^1\backslash \{(\mathbf{0},\mathbf{0},\mathbf{0})\}\ |\ S_{\omega,\mathbf{c}}(\Psi)< \mu_{\omega,\mathbf{c}},\ N(\Psi)> -2\mu_{\omega,\mathbf{c}}\right\},\\
   \mathcal{B}^{-}_{\omega,\mathbf{c}}&:=\left\{\Psi\in \mathcal{H}^1\backslash \{(\mathbf{0},\mathbf{0},\mathbf{0})\}\ |\ S_{\omega,\mathbf{c}}(\Psi)< \mu_{\omega,\mathbf{c}},\ N(\Psi)< -2\mu_{\omega,\mathbf{c}}\right\}.
 \end{split}
\end{equation}
The following Proposition will be used to prove the 
stability of ground states set in next subsection. 
\begin{proposition}\label{sol_AB}
Let $d\in \{1,2,3\}$ and $\alpha,\beta,\gamma>0$. 
Assume that $(\omega,\mathbf{c})\in \R\times \R^d$ 
satisfies $\omega >\frac{\sigma |\mathbf{c}|^2}{4}$.
Then, $\mathcal{A}^{+}_{\omega,\mathbf{c}}=\mathcal{B}^{+}_{\omega,\mathbf{c}}$ and 
$\mathcal{A}^{-}_{\omega,\mathbf{c}}=\mathcal{B}^{-}_{\omega,\mathbf{c}}$ hold, 
where $\mathcal{A}^{+}_{\omega,\mathbf{c}}$ and $\mathcal{A}^{-}_{\omega,\mathbf{c}}$ 
are defined in {\rm (\ref{A_pm_def})}.
\end{proposition}
\begin{proof}
We first prove $\mathcal{A}^{+}_{\omega,\mathbf{c}}= \mathcal{B}^{+}_{\omega,\mathbf{c}}$. 
Let $U\in \mathcal{A}^{+}_{\omega,\mathbf{c}}$. 
Because $S_{\omega,\mathbf{c}}(U)<\mu_{\omega,\mathbf{c}}$ and $K_{\omega,\mathbf{c}}(U)>0$ hold 
we obtain $
N(U)>-2\mu_{\omega,\mathbf{c}}$
by (\ref{SKL_eq4}). 
This says $U\in \mathcal{B}^{+}_{\omega,\mathbf{c}}$. 
Conversely, let $V\in \mathcal{B}^{+}_{\omega,\mathbf{c}}$. 
We assume $K_{\omega,\mathbf{c}}(V)\le 0$. 
Then, by Proposition~\ref{mu_prop} (ii), we have $L_{\omega,\mathbf{c}}(V)\ge 6\mu_{\omega,\mathbf{c}}$. 
Since $N(V)>-2\mu_{\omega,\mathbf{c}}$, this and (\ref{SKL_eq}) imply
\[
S_{\omega,\mathbf{c}}(V)=\frac{1}{2}L_{\omega,\mathbf{c}}(V)+N(V)>\mu_{\omega,\mathbf{c}},
\]
which contradicts to $V\in \mathcal{B}^{+}_{\omega,\mathbf{c}}$. 
Therefore, we get $K_{\omega,\mathbf{c}}(V)>0$, which implies $V\in \mathcal{A}^{+}_{\omega,\mathbf{c}}$.

Next, we prove $\mathcal{A}^{-}_{\omega,\mathbf{c}}= \mathcal{B}^{-}_{\omega,\mathbf{c}}$. 
Let $U\in \mathcal{A}^{-}_{\omega,\mathbf{c}}$.
Then, $K_{\omega,\mathbf{c}}(U)<0$ holds and 
we have $L_{\omega,\mathbf{c}}(U)>6\mu_{\omega,\mathbf{c}}$ by Proposition~\ref{mu_prop} (ii). 
Therefore, we obtain
\[
N(U)=\frac{1}{3}(K_{\omega,\mathbf{c}}(U)-L_{\omega,\mathbf{c}}(U))<-2\mu_{\omega,\mathbf{c}}.
\]
This says $U\in \mathcal{B}^{-}_{\omega,\mathbf{c}}$. 
Conversely, let $V\in \mathcal{B}^{-}_{\omega,\mathbf{c}}$. Then, 
by (\ref{SKL_eq4}) and $S_{\omega,\mathbf{c}}<\mu_{\omega,\mathbf{c}}$, we have
\[
K_{\omega,\mathbf{c}}(V)=2S_{\omega,\mathbf{c}}(V)+N(V)<0.
\]
This says $V\in \mathcal{A}^{-}_{\omega,\mathbf{c}}$.
\end{proof}
\begin{remark}
Let $\alpha$, $\beta$, $\gamma>0$, $\omega>\frac{\sigma |c|^2}{4}$, and $U_0\in \mathcal{B}^{+}_{\omega,\mathbf{c}}$. 
Then, we can obtain the global solution to (\ref{NLS}) with $U|_{t=0}=U_0$ 
by Theorem~\ref{gwp} and Proposition~\ref{sol_AB}. 
We note that $U_0\in \mathcal{B}^{+}_{\omega,\mathbf{c}}$ satisfies either
\[
{\rm (i)}\ -2\mu_{\omega,\mathbf{c}}<N(U_0)<0\ \ \ \ {\rm or}\ \ \ \ {\rm (ii)}\ N(U_0)\ge 0\ \ (then\ E(U_0)>0). 
\]
While the case (i) corresponds to the focusing case with small nonlinear effect, the case (ii) corresponds to the defocusing case. 
\end{remark}
\subsection{Proof of stability of ground states set}
\label{proof of stability}
In this subsection, we write $\mu_{\omega,\mathbf{c}}=\mu (\omega,\mathbf{c})$. 
For fixed $(\omega, \mathbf{c})\in \R\times \R^d$ with $\omega >\frac{\sigma |\mathbf{c}|^2}{4}$, we define the function $h=h_{\omega,\mathbf{c}}$ on $\bigl(-\infty,\sqrt{\omega}\bigl)$ as 
\[
h(\tau ):=
\mu \bigg((\sqrt{\omega}-\tau)^2,\frac{\mathbf{c}}{\sqrt{\omega}}(\sqrt{\omega}-\tau)\bigg)
=(\sqrt{\omega}-\tau)^{4-d}
\mu \left(1,\frac{\mathbf{c}}{\sqrt{\omega}}\right), 
\]
where to obtain the second equality, we have used Proposition~\ref{mu_scale}. 
Namely, the function $h$ is a restriction of $\mu$ 
on a so-called scaling curve
\[
\mathcal{C}_{\omega,\mathbf{c}}
:=\left\{\left.\left((\sqrt{\omega}-\tau)^2,\frac{\mathbf{c}}{\sqrt{\omega}}(\sqrt{\omega}-\tau )\right)\ \right|\ \tau < \sqrt{\omega}\ \right\}
\]
We note that
\[
\begin{split}
h'(\tau)&=-(4-d)(\sqrt{\omega}-\tau)^{3-d}\mu \left(1,\frac{\mathbf{c}}{\sqrt{\omega}}\right),\\
h''(\tau)&=(3-d)(4-d)(\sqrt{\omega}-\tau)^{2-d}\mu \left(1,\frac{\mathbf{c}}{\sqrt{\omega}}\right),
\end{split}
\]
and $h$ is strictly decreasing on $\bigl(-\infty,\sqrt{\omega}\bigl)$. By using these identities and (\ref{M_scale_eq}), the following lemma is obtained . 
\begin{lemma}\label{h_deriv_0_val}
    Let $d\in \{1,2\}$ and $\alpha,\beta,\gamma>0$. 
Assume that $(\omega,\mathbf{c})\in \R\times \R^d$ 
satisfies $\omega >\frac{\sigma |\mathbf{c}|^2}{4}$.
If $\Phi\in \mathcal{M}_{\omega,\mathbf{c}}$, then 
it holds that
\[
\begin{split}
h(0)&=\mu (\omega,\mathbf{c})=S_{\omega,\mathbf{c}}(\Phi),\\
h'(0)&=-\frac{1}{\sqrt{\omega}}(2\omega Q(\Phi)+\mathbf{c}\cdot \mathbf{P}(\Phi)),\\
h''(0)&=\frac{3-d}{\omega}(2\omega Q(\Phi)+\mathbf{c}\cdot \mathbf{P}(\Phi)). 
\end{split}
\]
\end{lemma}
For $\eta>0$, we define the function $F=F_{\omega,\mathbf{c},\eta}$ on $\bigl(-\infty,\sqrt{\omega}\bigl)$ as
\[
    F(\tau):=\inf_{\Phi\in \mathcal{M}_{\omega,\mathbf{c}}^*(\eta)}
\left(\frac{h''(\tau)}{2}-Q(\Phi)\right). 
    \]
\begin{lemma}\label{F_con_0}
Let $d\in \{1,2\}$ and $\alpha,\beta,\gamma>0$. 
Assume that $(\omega,\mathbf{c})\in \R\times \R^d$ 
satisfies $\omega >\frac{\sigma |\mathbf{c}|^2}{4}$.
There exists $\tau_*=\tau_*(\omega,\mathbf{c},\eta)>0$ 
such that $F(\tau)\ge \frac{\eta}{4\omega}$ holds for any $\tau \in (-\tau_*,\tau_*)$. 
\end{lemma}
\begin{proof}
    Because the definition of the function $h$ is 
    independent of $\Phi\in \mathcal{M}_{\omega,\mathbf{c}}^*(\eta)$, 
    we have
    \[
    F(\tau)=\frac{h''(\tau)}{2}-\sup_{\Phi\in \mathcal{M}_{\omega,\mathbf{c}}^*(\eta)}Q(\Phi). 
    \]
    In particular, $F$ is continuous 
    because $h''(\tau)$ is polynomial. 
    On the other hand, we have
    \[
\frac{h''(0)}{2}-Q(\Phi)=\frac{1}{2\omega}\left\{
(4-2d)\omega Q(\Phi)+(3-d)\mathbf{c}\cdot \mathbf{P}(\Phi))\right\}\ge \frac{\eta}{2\omega}\ (>0)
\]
for any $\Phi\in \mathcal{M}_{\omega,\mathbf{c}}^*(\eta)$
by Lemma~\ref{h_deriv_0_val}. 
This implies $F(0)\ge \frac{\eta}{2\omega}$. 
Therefore, we obtain the conclusion by the continuity of $F$ on $\tau=0$. 
\end{proof}
For $\tau_0\in \Bigl(0,\sqrt{\omega}\Bigl)$, we set
\[
\omega_{\pm}=\omega_{\pm}(\tau_0):=(\sqrt{\omega}\pm \tau_0)^2,\ \ 
\mathbf{c}_{\pm}=\mathbf{c}_{\pm}(\tau_0):=\frac{\mathbf{c}}{\sqrt{\omega}}(\sqrt{\omega} \pm \tau_0),
\]
where the double-signs correspond. The following proposition plays an important role 
to prove the stability result. 
\begin{proposition}\label{sta_prop_B}
 Let $d\in \{1,2\}$, $\alpha,\beta,\gamma>0$, and $\eta >0$. 
Assume that  
$(\omega,\mathbf{c})\in \R\times \R^d$ 
satisfies
$\omega >\frac{\sigma |\mathbf{c}|^2}{4}$
and 
$\mathcal{M}_{\omega,\mathbf{c}}^*(\eta)\ne \emptyset$. 
Then for any $\tau_0\in (0,\tau_*)$, 
there exists $\delta =\delta (\tau_0,\tau_*,\omega,\mathbf{c},\eta)>0$ such that 
if
$U_0\in \mathcal{H}^1$ satisfies
\begin{equation}
\label{Condition on initial}
\inf_{\Phi \in \mathcal{M}_{\omega,\mathbf{c}}^*(\eta)}\|U_0-\Phi\|_{\mathcal{H}^1}<\delta,
\end{equation}
then it holds $U_0\in \mathcal{B}^{+}_{\omega_+,\mathbf{c}_+}\cap \mathcal{B}^{-}_{\omega_-,\mathbf{c}_-}$,  
where $\tau_*>0$ is given in Lemma~\ref{F_con_0} and
$\mathcal{B}^{\pm}_{\omega_{\pm},\mathbf{c}_{\pm}}$ are defined in (\ref{gwp_set_B}). 
Furthermore, such $\delta$ tends to $0$ as $\tau_0\rightarrow 0$. 
\end{proposition}
\begin{proof}
We assume that $U_0\in \mathcal{H}^1$ satisfies (\ref{Condition on initial})
with sufficiently small $\delta\in (0,1)$, which will be chosen later. 
By the definition of the infimum, there exists 
$\Phi_{\omega,\mathbf{c}}\in \mathcal{M}_{\omega,\mathbf{c}}^*(\eta)$ such that
\[
\|U_0-\Phi_{\omega,\mathbf{c}}\|_{\mathcal{H}^1}<2\delta. 
\]
Therefore, we have
\begin{equation}\label{u0_phi_diff_est}
\begin{split}
|Q(U_0)-Q(\Phi_{\omega,\mathbf{c}})|&\lesssim
(\|U_0\|_{\mathcal{H}^1}+\|\Phi_{\omega,\mathbf{c}}\|_{\mathcal{H}^1})
\|U_0-\Phi_{\omega,\mathbf{c}}\|_{\mathcal{H}^1}\lesssim \delta,\\
|P(U_0)-P(\Phi_{\omega,\mathbf{c}})|&\lesssim
(\|U_0\|_{\mathcal{H}^1}+\|\Phi_{\omega,c}\|_{\mathcal{H}^1})
\|U_0-\Phi_{\omega,\mathbf{c}}\|_{\mathcal{H}^1}\lesssim \delta,\\
|N(U_0)-N(\Phi_{\omega,\mathbf{c}})|&\lesssim 
(\|U_0\|_{\mathcal{H}^1}^2
+\|\Phi_{\omega,\mathbf{c}}\|_{\mathcal{H}^1}^2)
\|U_0-\Phi_{\omega,\mathbf{c}}\|_{\mathcal{H}^1}
\lesssim \delta
\end{split}
\end{equation}
and
\begin{equation}\label{u0_phi_diff_est_2}
|S_{\omega,c}(U_0)-S_{\omega,\mathbf{c}}(\Phi_{\omega,\mathbf{c}})|\lesssim 
(\|U_0\|_{\mathcal{H}^1}+\|\Phi_{\omega,\mathbf{c}}\|_{\mathcal{H}^1}
+\|U_0\|_{\mathcal{H}^1}^2+\|\Phi_{\omega,\mathbf{c}}\|_{\mathcal{H}^1}^2)
\|U_0-\Phi_{\omega,\mathbf{c}}\|_{\mathcal{H}^1}\lesssim \delta. 
\end{equation}
Here we note that the above implicit constants do not depend on 
$U_0$ and $\Phi_{\omega,\mathbf{c}}$. 
Indeed, by Proposition~\ref{L_bdd} and (\ref{SKL_eq}), there exists $C=C(\omega,\mathbf{c})$ such that
\[
\|\Phi_{\omega,\mathbf{c}}\|_{\mathcal{H}^1}^2
\le CL_{\omega,\mathbf{c}}(\Phi_{\omega ,\mathbf{c}})=6C\mu_{\omega,\mathbf{c}}.
\]
This implies that
\[
\|U_0\|_{\mathcal{H}^1}+\|\Phi_{\omega,\mathbf{c}}\|_{\mathcal{H}^1}
\le
\|U_0-\Phi_{\omega,\mathbf{c}}\|_{\mathcal{H}^1}
+2\|\Phi_{\omega,\mathbf{c}}\|_{\mathcal{H}^1}
\le
\delta +2\sqrt{6C\mu_{\omega,\mathbf{c}}}.
\]

We divide the proof into the following two steps:\\
\noindent {\bf Step\ 1}.\ 
We first prove that for any $\tau_0 \in (0,\tau_*)$,
\begin{equation}\label{S_mu_ep}
S_{\omega_{\pm},\mathbf{c}_{\pm}}(U_0)<\mu (\omega_{\pm},\mathbf{c}_{\pm}).
\end{equation} 
By (\ref{u0_phi_diff_est}), 
(\ref{u0_phi_diff_est_2}), 
and Lemma~\ref{h_deriv_0_val}, 
we obtain
\[
\begin{split}
S_{\omega_{\pm},\mathbf{c}_{\pm}}(U_0)
&=E(U_0)+\omega_{\pm}Q(U_0)+\mathbf{c}_{\pm}\cdot \mathbf{P}(U_0)\\ 
&=S_{\omega,\mathbf{c}}(U_0)
\pm \frac{\tau_0}{\sqrt{\omega}}\left(2\omega Q(U_0)+\mathbf{c}\cdot \mathbf{P}(U_0)\right)
+\tau_0^2Q(U_0)\\
&=S_{\omega,\mathbf{c}}(\Phi_{\omega,\mathbf{c}})
\pm \frac{\tau_0}{\sqrt{\omega}}
\left(2\omega Q(\Phi_{\omega,\mathbf{c}})+\mathbf{c}\cdot \mathbf{P}(\Phi_{\omega,\mathbf{c}})\right)
+\tau_0^2Q(\Phi_{\omega,\mathbf{c}})+O(\delta)\\
&=h(0)\mp \tau_0 h'(0)+\tau_0^2Q(\Phi_{\omega,\mathbf{c}})+O(\delta). 
\end{split}
\]
On the other hand, by the Taylor expansion, there exists $\theta \in (-\tau_0, \tau_0)$ such that
\[
h(\mp \tau_0)=h(0)\mp \tau_0 h'(0)+\frac{\tau_0^2}{2}h''(\theta). 
\]
Therefore, we obtain
\[
\begin{split}
S_{\omega_{\pm},\mathbf{c}_{\pm}}(U_0)
=h(\mp \tau_0)-\tau_0^2\left(\frac{h''(\theta)}{2}-Q(\Phi_{\omega,\mathbf{c}})\right)+O(\delta)
\le h(\mp \tau_0)-\tau_0^2F(\theta)+O(\delta). 
\end{split}
\]
Since $|\theta|<\tau_0<\tau_*$, by Lemma~\ref{F_con_0}, we have $F(\theta)\ge \frac{\eta}{4\omega}$. Here we choose $\delta=\delta(\tau_0,\omega,\eta)>0$ as
\[
-\frac{\tau_0^2\eta}{4\omega}+O(\delta)<0.
\]
Then we have
\[
S_{\omega_{\pm},\mathbf{c}_{\pm}}(U_0)
<h(\mp \tau_0).
\]
Therefore we get (\ref{S_mu_ep}) by the definition of $h$. 
We also note that $\delta$ tends to $0$ as $\tau_0\rightarrow 0$. 
\\

\noindent {\bf Step\ 2}.\ Next, we prove that for any $\tau_0\in (0,\tau_*)$,
\begin{equation}\label{mu_N_pm_est}
-2\mu (\omega_+,\mathbf{c}_+)
<N(U_0)<-2\mu (\omega_-,\mathbf{c}_-). 
\end{equation}
Because
$h$ is a decreasing function on $\Bigl(-\infty,\sqrt{\omega}\Bigl)$, 
it holds that
\[
h(\tau_*)<h(0)<h(-\tau_*). 
\]
Since $K_{\omega,\mathbf{c}}(\Phi_{\omega,\mathbf{c}})=0$, by using (\ref{SKL_eq4}), we have
\begin{equation}\label{N_range}
-2h(-\tau_*)<N(\Phi_{\omega,\mathbf{c}})<-2h(\tau_*).
\end{equation}
Therefore, by the third inequality of (\ref{u0_phi_diff_est}), there exists $C=C(\omega,\mathbf{c})>0$ such that
\[
C\delta -2h(-\tau_*)
<N(U_0)<C\delta -2h(\tau_*).
\] 
By choosing $\delta=\delta (\tau_0,\tau_*,\omega,\mathbf{c})>0$ further smaller as
\[
C\delta<2\min\{h(\tau_0)-h(\tau_*),h(-\tau_*)-h(-\tau_0)\}(>0), 
\]
we have 
\[
-2h(-\tau_0)
<N(U_0)<-2h(\tau_0),
\]
which is equivalent to (\ref{mu_N_pm_est}). This completes the proof of the proposition.
\end{proof}
For $n\in \N$, we set
\[
\omega_{n\pm}:=\omega_{\pm}(1/n)=\left(\sqrt{\omega}\pm \frac{1}{n}\right)^2,\ \ 
\mathbf{c}_{n\pm}:=\mathbf{c}_{\pm}(1/n)=\frac{\mathbf{c}}{\sqrt{\omega}}\left(\sqrt{\omega} \pm \frac{1}{n}\right). 
\]
Let $\tau_*$ be given in Lemma \ref{F_con_0}. Then by Proposition~\ref{sta_prop_B}, 
for $n>\frac{1}{\tau_*}$, 
there exists $\delta_n=\delta(1/n,\tau_{*},\eta,\omega,\mathbf{c})>0$ 
satisfying $\delta_n\rightarrow 0$ as $n\rightarrow \infty$, such that if $U_0\in \mathcal{H}^1$ 
satisfies
\[
\inf_{\Phi \in \mathcal{M}_{\omega,\mathbf{c}}^*(\eta)}\|U_0-\Phi\|_{\mathcal{H}^1}<\delta_n, 
\]
then $U_0\in \mathcal{B}^{+}_{\omega_{n+},\mathbf{c}_{n+}}\cap \mathcal{B}^{-}_{\omega_{n-},\mathbf{c}_{n-}}$. 

Now we give a proof of the stability of the ground-states set $\mathcal{M}_{\omega,\mathbf{c}}^*(\eta)$. 
\begin{proof}[Proof of Theorem~\ref{stability} ]
We prove the stability by contradiction. We assume that the ground-states set $\mathcal{M}_{\omega,\mathbf{c}}^*(\eta)$ is unstable. 
Then, there exists $\epsilon >0$ such that the following holds:

For each $n\in \N$ sufficiently large with $n>\frac{1}{\tau_*}$, 
there exists $U_{n,0}\in \mathcal{H}^1$
such that
\begin{equation}\label{stab_ineq_0}
    \inf_{\Phi \in \mathcal{M}_{\omega,\mathbf{c}}^*(\eta)}\|U_{n,0}-\Phi\|_{\mathcal{H}^1}<\delta_n, 
\end{equation}
and one of
the following {\rm (i)} or {\rm (ii)} are satisfied, where $\delta_n>0$ is given as above. 
\begin{enumerate}
    \item[{\rm (i)}] The time global solution to (\ref{NLS}) with 
    initial data $U_{n,0}$ does not exists. 
    \item[{\rm (ii)}] There exists a time global solution $U_n$ to (\ref{NLS}) with 
    initial data $U_{n,0}$ and $t_n>0$ 
    such that
   \begin{equation}\label{stab_ineq}
\inf_{\Phi\in \mathcal{M}_{\omega,\mathbf{c}}^*(\eta)}\|U_n(t_n)-\Phi\|_{\mathcal{H}^1}\ge \epsilon. 
\end{equation}
\end{enumerate}
Because $U_{n,0}\in \mathcal{B}^{+}_{\omega_{n+},\mathbf{c}_{n+}}\cap \mathcal{B}^{-}_{\omega_{n-},\mathbf{c}_{n-}}$ by (\ref{stab_ineq_0}) and Proposition \ref{sta_prop_B}, in particular $U_{n,0}\in \mathcal{B}^{+}_{\omega_{n+},\mathbf{c}_{n+}}$, 
there exists a solution to (\ref{NLS}) with 
    initial data $U_{n,0}$ globally in time by Theorem~\ref{gwp} 
    and Proposition~\ref{sol_AB}. 
Therefore, the case {\rm (i)} does not occur.  

Now we assume the case {\rm (ii)}. 
By Propositions~\ref{A_sol_inv} and 
~\ref{sol_AB}, 
we have $U_n(t_n)\in \mathcal{B}^{+}_{\omega_{n+},\mathbf{c}_{n+}}\cap \mathcal{B}^{-}_{\omega_{n-},\mathbf{c}_{n-}}$ 
because $U_{n,0}\in \mathcal{B}^{+}_{\omega_{n+},\mathbf{c}_{n+}}\cap \mathcal{B}^{-}_{\omega_{n-},\mathbf{c}_{n-}}$. 
This implies 
\[
-2h\left(-\frac{1}{n}\right)<N(U_n(t_n))<-2h\left(\frac{1}{n}\right). 
\]
Thus as $n\rightarrow \infty$ by the continuity of $h$, we obtain
\begin{equation}\label{stab_N_conv}
N(U_n(t_n))\ \rightarrow\ -2h(0)=-2\mu_{\omega,\mathbf{c}}. 
\end{equation} 
Furthermore, by (\ref{stab_ineq_0}) and the definition of the infimum, there exists $\Phi_{\omega,\mathbf{c},n}\in \mathcal{M}_{\omega,\mathbf{c}}^*(\eta)$ such that
\begin{equation}\label{stab_ineq2}
\|U_{n,0}-\Phi_{\omega,\mathbf{c},n}\|_{\mathcal{H}^1}<2\delta_n.
\end{equation}
Since $S_{\omega,\mathbf{c}}$ is a conserved quantity and $\Phi_{\omega,\mathbf{c},n}$ 
is a minimizer of $\mu_{\omega,\mathbf{c}}$, by (\ref{stab_ineq2}), we have
\begin{equation}\label{stab_S_conv}
|S_{\omega,\mathbf{c}}(U_n(t_n))-\mu_{\omega,\mathbf{c}}|
=|S_{\omega,\mathbf{c}}(U_{n,0})-S_{\omega,\mathbf{c}}
(\Phi_{\omega,\mathbf{c},n})|
\ \rightarrow\ 0
\end{equation}
as $n\rightarrow \infty$. 
By using (\ref{SKL_eq4}), 
(\ref{stab_N_conv}), and (\ref{stab_S_conv}), we get
\begin{equation}\label{stab_K_conv}
K_{\omega,c}(U_n(t_n))=N(U_n(t_n))+2S_{\omega,\mathbf{c}}(U_n(t_n))\ 
\rightarrow\ 0. 
\end{equation}
as $n\rightarrow \infty$. 
In addition, because $Q$ and $\mathbf{P}$ are conserved quantities, by (\ref{stab_ineq2}), we obtain
\[
\begin{split}
    G(U_n(t_n))&=(4-2d)\omega Q(U_n(t_n))+(3-d)\mathbf{c}\cdot \mathbf{P}(U_n(t_n))\\
    &\ =(4-2d)\omega Q(U_n(0))+(3-d)\mathbf{c}\cdot \mathbf{P}(U_n(0))\\
    &\ =(4-2d)\omega Q(\Phi_{\omega,\mathbf{c},n})
    +(3-d)\mathbf{c}\cdot \mathbf{P}(\Phi_{\omega,\mathbf{c},n})
    +O(\delta_n).
\end{split}
\]
These identities and the relation $\Phi_{\omega,\mathbf{c},n}\in \mathcal{M}_{\omega,\mathbf{c}}^*(\eta)$ 
imply
\begin{equation}\label{stab_add_conv}
\displaystyle \limsup_{n\rightarrow \infty}G(U_n(t_n))
=\displaystyle \limsup_{n\rightarrow \infty}G(\Phi_{\omega,\mathbf{c},n})\ge \eta.
\end{equation}
Therefore, by (\ref{stab_S_conv}), (\ref{stab_K_conv}), (\ref{stab_add_conv}), and
Proposition~\ref{GS_exist}, 
there exist $\{y_n\}\subset \R^d$ and $V\in \mathcal{M}_{\omega,\mathbf{c}}^*(\eta)$ 
such that $\{U_n(t_n,\cdot -y_n)\}$ has a subsequence 
which converges strongly to $V$ in $\mathcal{H}^1$. 
Therefore, by setting $\Phi_n:=V(\cdot +y_n)$ 
and taking subsequence, 
we obtain
\[
\|U_n(t_n)-\Phi_n\|_{\mathcal{H}^1}\ \rightarrow\ 0. 
\]
This contradicts to (\ref{stab_ineq}) 
since $\Phi_n\in \mathcal{M}_{\omega,\mathbf{c}}^*(\eta)$. 
\end{proof}
\subsection{Stability of ground state with small speed for $d=1$}
To prove Corollary~\ref{stab_1d_result}, 
we introduce $\widetilde{\mathcal{M}}_{\omega,c_*}$ for $c_*>0$ defined by
\[
\widetilde{\mathcal{M}}_{\omega,c_*}
:=\bigcup_{|c|\le c_*}\mathcal{M}_{\omega, c},
\]
and prepare the following lemma. 
\begin{lemma}\label{1d_stab_lemm}
Let $d=1$ and $\alpha,\beta,\gamma>0$. 
Assume that $(\omega,c)\in \R\times \R$ 
satisfies $\omega >\frac{\sigma c^2}{4}$. Then the following holds:
\begin{itemize}
    \item[{\rm (i)}] There exists $c_1=c_1(\omega)>0$ such that $\widetilde{\mathcal{M}}_{\omega,c_1(\omega)}$ 
    is bounded in $\mathcal{H}^1$. Namely, it holds that
    \[
    A_{\omega}:=\sup_{\Phi \in \widetilde{\mathcal{M}}_{\omega,c_1(\omega)}}
    \|\Phi\|_{\mathcal{H}^1}\in (0,\infty).
    \]
    \item[{\rm (ii)}] There exists $c_2=c_2(\omega)>0$ 
    such that it holds that
    \[B_{\omega}:=\inf_{\Phi \in \widetilde{\mathcal{M}}_{\omega,c_2(\omega)}}\omega Q(\Phi)>0.
    \]
\end{itemize}
\end{lemma}
\begin{proof}
We first prove {\rm (i)}. By Theorem~\ref{GS_exist_th}, there exists an element $\Phi_{\omega}\in \mathcal{M}_{\omega,0}$. Thus we see that $A_{\omega}>0$. 
Because $K_{\omega, 0}(\Phi_{\omega})=0$ and 
$N(\Phi_{\omega})=-2\mu_{\omega,0}$ by (\ref{SKL_eq4}), for $\lambda \in \R$, we have
    \[
\begin{split}
    K_{\omega,c}(\lambda \Phi_{\omega})
    &=2\lambda^2L(\Phi_{\omega})+3\lambda^3N(\Phi_{\omega})
    +2\lambda^2\omega Q(\Phi_{\omega})
    +2\lambda^2cP(\Phi_{\omega})\\
    &=\lambda^2K_{\omega,0}(\Phi_{\omega})
    +\lambda^2\{3(\lambda -1)N(\Phi_{\omega})+2cP(\Phi_{\omega})\}\\
    &=2\lambda^2\{-3(\lambda -1)\mu_{\omega,0}+cP(\Phi_{\omega})\}.
\end{split}
    \] 
Since $\mu_{\omega ,0}> 0$ by Proposition~\ref{mu_prop}, 
we can choose $\lambda$ as $\lambda=\lambda_0$, where $\lambda_0:=1+\frac{1}{3\mu_{\omega,0}}cP(\Phi_{\omega})$. Then, it holds that $K_{\omega,c}(\lambda_0 \Phi_{\omega})=0$. 
Because
\[
\frac{1}{3\mu_{\omega,0}}|cP(\Phi_{\omega})|
\le \frac{1}{3\mu_{\omega,0}}|c|\cdot 3\|\Phi_{\omega}\|_{\mathcal{H}^1}^2
=\frac{|c|}{\mu_{\omega,0}}\|\Phi_{\omega}\|_{\mathcal{H}^1}^2, 
\]
if $|c|< \mu_{\omega,0}\|\Phi_{\omega}\|_{\mathcal{H}^1}^{-2}$ 
holds, then we have $0<\lambda_0 <2$. 
Therefore, by the definitions of $\mu_{\omega,c}$ and $\lambda_0$, 
we get
\[
\begin{split}
\mu_{\omega,c}
&\le S_{\omega,c}(\lambda_0 \Phi_{\omega})\\
&=\lambda_0^2L(\Phi_{\omega})+\lambda_0^3N(\Phi_{\omega})
+\lambda_0^2\omega Q(\Phi_{\omega})+\lambda_0^2cP(\Phi_{\omega})\\
&=\lambda_0^2S_{\omega,0}(\Phi_{\omega})
+\lambda_0^2\{(\lambda_0-1)N(\Phi_{\omega})+cP(\Phi_{\omega})\}\\
&=\lambda_0^3\mu_{\omega,0}<8\mu_{\omega,0}. 
\end{split}
\]
On the other hand, by (\ref{SKL_eq}) and 
Proposition~\ref{L_bdd}, 
it holds that
\begin{equation}\label{mu_low_norm}
\mu_{\omega,c}=\frac{1}{6}L_{\omega,c}(\Phi)
\ge C_{\omega}(4\omega -\sigma c^2)\|\Phi\|_{\mathcal{H}^1}^2
\end{equation}
for any $\Phi\in \mathcal{M}_{\omega,c}$, 
where $C_{\omega}>0$ is a constant 
which depends on $\omega$ but does not depend on $c$.

By the above argument, we put
\[
c_{1}=c_1(\omega):=\min\left\{\frac{\mu_{\omega,0}}{2\|\Phi_{\omega}\|_{\mathcal{H}^1}^2}, \sqrt{\frac{2\omega}{\sigma}}\right\}. 
\]
Then, for any $\Phi \in \mathcal{M}_{\omega,c}$ with $|c|\le c_1$, 
we have
\[
\|\Phi\|_{\mathcal{H}^1}^2
\le \frac{\mu_{\omega,c}}{C_{\omega}(4\omega -\sigma c^2)}
< \frac{4}{C_{\omega}\omega}\mu_{\omega,0}. 
\]

Next, we prove {\rm (ii)}. Let $\Phi \in \mathcal{M}_{\omega,c}$. 
By (\ref{SKL_eq4}), the H\"older inequality, 
and the Sobolev inequality, we obtain
\[
\mu_{\omega,c}
=-\frac{1}{2}N(\Phi)
\le C\|\Phi\|_{\mathcal{H}^1}^3, 
\]
where the constant $C>0$ does not depend on 
$\omega$ and $c$. 
Therefore, by (\ref{mu_low_norm}), we have
\[
\frac{C_{\omega}(4\omega -\sigma c^2)}{C}\le \|\Phi\|_{\mathcal{H}^1}
\]
because $\Phi \ne 0$ in $\mathcal{H}^1$. 
In particular, if $|c|\le \sqrt{\frac{2\omega}{\sigma}}$, 
then it holds that
$\frac{2C_{\omega}\omega}{C}\le \|\Phi\|_{\mathcal{H}^1}$. This implies that
\[
L(\Phi)+\omega Q(\Phi)
\ge C'(1+\omega)\|\Phi\|_{\mathcal{H}^1}^2
\ge \frac{4C'C_{\omega}^2\omega^2(1+\omega)}{C^2}
=:\widetilde{B}_{\omega},
\]
where $C'=\min\{1,\frac{\alpha}{2},\frac{\beta}{2},\frac{\gamma}{2}\}$. 
Hence, at least one of
\[
2\omega Q(\Phi)\ge \widetilde{B}_{\omega}\ \ \ \ 
{\rm or}\ \ \ \ 2L(\Phi)\ge \widetilde{B}_{\omega}
\]
holds. If the former estimate holds, 
then by setting $c_2=c_2(\omega):=\sqrt{\frac{2\omega}{\sigma}}>0$, the proof of (ii) is completed 
as $B_{\omega}=\frac{\widetilde{B}_{\omega}}{2}$. 
On the other hand, if the latter estimate holds, 
then by Proposition~\ref{LN_iden} with $d=1$ and 
$K_{\omega,c}(\Phi)=0$, the estimates $\omega Q(\Phi)=L(\Phi)\ge \frac{\widetilde{B}_{\omega}}{2}$ hold.
\end{proof}
Corollary~\ref{stab_1d_result} follows from 
Theorem~\ref{stability} and the following proposition. 
\begin{proposition}[Stability of the set $\mathcal{M}_{\omega,c}$]
\label{1d_stability}
Let $d=1$ and $\alpha,\beta,\gamma>0$. 
Assume that 
$(\omega,c)\in \R\times \R$ 
satisfies
$\omega >\frac{\sigma c^2}{4}$. 
Then, there exist
$\eta_0=\eta_0(\omega)>0$ and
$c_0=c_0(\omega)>0$ 
such that if $|c|\le c_0$, 
then $\mathcal{M}_{\omega,c}^*(\eta_0)=\mathcal{M}_{\omega,c}(\neq \emptyset)$ 
holds. Namely, the estimate
\[
\omega Q(\Phi)+cP(\Phi)\ge \eta_0
\]
holds for any $\Phi \in \mathcal{M}_{\omega,c}$ 
with $|c|\le c_0$. 
\end{proposition}
\begin{proof}
Let $c_1=c_1(\omega)>0$ and $c_2=c_2(\omega)>$ 
are the constants given in Lemma~\ref{1d_stab_lemm}. 
If $|c|\le \min\{c_1,c_2\}$, 
then by Lemma~\ref{1d_stab_lemm}, the estimates $\omega Q(\Phi)\ge B_{\omega}$ and $|P(\Phi)|\le 3\|\Phi\|_{\mathcal{H}^1}^2\le 3A_{\omega}^2
$ hold for any $\Phi \in \mathcal{M}_{\omega,c}$, which implies that $\omega Q(\Phi)+cP(\Phi)
\ge B_{\omega}-3|c|A_{\omega}^2$. 
Here we set
\[
\eta_{0}=\eta_0(\omega):=\frac{B_{\omega}}{2},\ \ 
c_0=c_0(\omega):=\min\left\{c_1(\omega),c_2(\omega),
\frac{B_{\omega}}{6A_{\omega}^2}
\right\}>0.
\]
Then the estimate $\omega Q(\Phi)+cP(\Phi)
\ge \eta_0$ holds for any $\Phi \in \mathcal{M}_{\omega,c}$ with $|c|\le c_0$. 
\end{proof}

\appendix

\section{Exponential decay of solution to the stationary Schr\"odinger system}
In this appendix we prove that solution to the elliptic system (\ref{ellip}) decays exponentially at $|x|\rightarrow \infty$.
\begin{proposition}
Let $d\in \{1,2,3\}$ and $\alpha,\beta,\gamma>0$. 
Assume that $(\omega,\mathbf{c})\in \R\times \R^d$ satisfies
$\omega >\frac{\sigma |\mathbf{c}|^2}{4}$.
If $(\varphi_1,\varphi_2,\varphi_3)\in \mathcal{H}^1$ be a weak solution to {\rm (\ref{ellip})}, 
then we have
\begin{equation}\label{weaksol_decay}
\sum_{j=1}^3\int_{\R^d}e^{p|x|}
\left(|\varphi_j(x)|^2+|\nabla \varphi_j (x)|^2\right)dx<\infty
\end{equation}
for
\begin{equation}\label{kappa_cond}
0<p<\sqrt{4\omega \sigma_0}\left(1-\sqrt{\frac{\sigma}{4\omega}}|c|\right),
\end{equation}
where
\[
\sigma_0:=\min\left\{\frac{2}{\alpha},\frac{1}{\beta},\frac{1}{\gamma}\right\}. 
\]
\end{proposition}
\begin{proof}
For $\epsilon >0$, we define $\theta_{\epsilon}:\R^d\rightarrow \R_{> 0}$ as
\[
\theta_{\epsilon}(x):=\exp \left(\frac{p\sqrt{1+|x|^2}}{1+\epsilon \sqrt{1+|x|^2}}\right). 
\]
We note that $\theta_{\epsilon}$ and $\nabla \theta_{\epsilon}$ are bounded on $\R^d$. 
Indeed, by a simple calculation, we can see that 
$0\le \theta_{\epsilon}(x)\le e^{\frac{p}{\epsilon}}$ and
\begin{equation}\label{theta_deriva}
    |\nabla \theta_{\epsilon}(x)|\le p\theta_{\epsilon}(x) 
\end{equation}
for any $x\in \R^d$. 
We define the real-valued inner product $(\cdot ,\cdot)$ on $(L^2(\R^d))^d$ as
$(f,g):={\rm Re}(f,g)_{L^2(\R^d)}$ for $f,g\in (L^2(\R^d))^d$. 
Because $(\varphi_1,\varphi_2,\varphi_3)$ be a weak solution to (\ref{ellip}), we have $\varphi_1$, $\varphi_2$, $\varphi_3\in \left(\bigcap_{m=1}^{\infty}H^m(\R^d)\right)^d$ and
\[
\begin{split}
0&=(-\alpha \Delta \varphi_1+2\omega \varphi_1+i(\mathbf{c}\cdot \nabla) \varphi_1
-(\nabla \cdot \varphi_3)\varphi_2,\theta_{\epsilon}\varphi_1),\\
0&=(-\beta \Delta \varphi_2+\omega \varphi_2+i(\mathbf{c}\cdot \nabla) \varphi_2
-(\nabla \cdot \overline{\varphi_3})\varphi_1,\theta_{\epsilon}\varphi_2),\\
0&=(-\gamma \Delta \varphi_3+\omega \varphi_3+i(\mathbf{c}\cdot \nabla) \varphi_3
+\nabla (\varphi_1\cdot \overline{\varphi_2}),\theta_{\epsilon}\varphi_3)
\end{split}
\]
by Corollary \ref{phi_regularity}. 
By using the integration by parts 
for the first terms of right hand side, 
these equations can be written as
\[
\begin{split}
0&=\alpha (\nabla \varphi_1,\theta_{\epsilon}\nabla \varphi_1)
+\alpha I_{1,1}
+2\omega (\varphi_1,\theta_{\epsilon}\varphi_1)
+I_{2,1}
-J_1,\\
0&=\beta (\nabla \varphi_2,\theta_{\epsilon}\nabla \varphi_2)
+\beta I_{1,2}
+\omega (\varphi_2,\theta_{\epsilon}\varphi_2)
+I_{2,2}
-J_2,\\
0&=\gamma (\nabla \varphi_3,\theta_{\epsilon}\nabla \varphi_3)
+\gamma I_{1,3}
+\omega (\varphi_3,\theta_{\epsilon}\varphi_3)
+I_{2,3}
-J_3, 
\end{split}
\]
where 
\[
I_{1,j}=
\sum_{k=1}^d(\nabla \varphi_{j}^{(k)}\cdot \nabla \theta_{\epsilon},\varphi_{j}^{(k)}),\ \ \ 
I_{2,j}=(i(\mathbf{c}\cdot \nabla)\varphi_j,\theta_{\epsilon}\varphi_j),\ \ 
(\varphi_j=(\varphi_j^{(1)},\cdots,\varphi_j^{(d)}))
\]
for $j=1,2,3$
and
\[
J_1:=((\nabla \cdot \varphi_3)\varphi_2,\theta_{\epsilon}\varphi_1),\ \ 
J_2:=((\nabla \cdot \overline{\varphi_3})\varphi_1,\theta_{\epsilon}\varphi_2),\ \ 
J_3:=-(\nabla (\varphi_1\cdot \overline{\varphi_2}),\theta_{\epsilon}\varphi_3). 
\]
Then, we obtain
\[
\begin{split}
J_1&=\alpha \|\theta_{\epsilon}^{\frac{1}{2}}|\nabla \varphi_1|\|_{L^2}^2
+2\omega \|\theta_{\epsilon}^{\frac{1}{2}}|\varphi_1|\|_{L^2}^2
+\alpha I_{1,1}+I_{2,1},\\
J_2&=\beta \|\theta_{\epsilon}^{\frac{1}{2}}|\nabla \varphi_2|\|_{L^2}^2
+\omega \|\theta_{\epsilon}^{\frac{1}{2}}|\varphi_2|\|_{L^2}^2
+\beta I_{1,2}+I_{2,2},\\
J_3&=\gamma \|\theta_{\epsilon}^{\frac{1}{2}}|\nabla \varphi_3|\|_{L^2}^2
+\omega \|\theta_{\epsilon}^{\frac{1}{2}}|\varphi_3|\|_{L^2}^2
+\gamma I_{1,3}+I_{2,3}. 
\end{split}
\]
By (\ref{theta_deriva}) and the Young inequality, we have
\[
|I_{1,j}|\le \sum_{k=1}^d\int_{\R^d}|\nabla \varphi_j^{(k)}(x)|p\theta_{\epsilon}(x)|\varphi_j^{(k)}(x)|dx
\le p\left(\frac{\delta_j}{2}\|\theta_{\epsilon}^{\frac{1}{2}}|\nabla \varphi_j|\|_{L^2}^2
+\frac{1}{2\delta_j}\|\theta_{\epsilon}^{\frac{1}{2}}|\varphi_j|\|_{L^2}^2\right)
\]
and
\[
|I_{2,j}|\le \int_{\R^d}|\mathbf{c}||\nabla \varphi_j(x)|\theta_{\epsilon}(x)|\varphi_j(x)|dx
\le |\mathbf{c}|\left(\frac{\eta_j}{2}\|\theta_{\epsilon}^{\frac{1}{2}}|\nabla \varphi_j|\|_{L^2}^2
+\frac{1}{2\eta_j}\|\theta_{\epsilon}^{\frac{1}{2}}|\varphi_j|\|_{L^2}^2\right), 
\]
where $\delta_j>0$, $\eta_j>0$ $(j=1,2,3)$ will be chosen later. 
Therefore, we get
\[
\begin{split}
    |J_1|&\ge \left(\alpha -\frac{\alpha p\delta_1}{2}-\frac{|\mathbf{c}|\eta_1}{2}\right)\|\theta_{\epsilon}^{\frac{1}{2}}|\nabla \varphi_1|\|_{L^2}^2
    +\left(2\omega -\frac{\alpha p}{2\delta_1}-\frac{|\mathbf{c}|}{2\eta_1}\right)\|\theta_{\epsilon}^{\frac{1}{2}}|\varphi_1|\|_{L^2}^2
    =:A_1\|\theta_{\epsilon}^{\frac{1}{2}}|\nabla \varphi_1|\|_{L^2}^2+B_1\|\theta_{\epsilon}^{\frac{1}{2}}|\varphi_1|\|_{L^2}^2,\\
    |J_2|&\ge \left(\beta -\frac{\beta p\delta_2}{2}-\frac{|\mathbf{c}|\eta_2}{2}\right)\|\theta_{\epsilon}^{\frac{1}{2}}|\nabla \varphi_2|\|_{L^2}^2
    +\left(\omega -\frac{\beta p}{2\delta_2}-\frac{|\mathbf{c}|}{2\eta_2}\right)\|\theta_{\epsilon}^{\frac{1}{2}}|\varphi_2|\|_{L^2}^2
    =:A_2\|\theta_{\epsilon}^{\frac{1}{2}}|\nabla \varphi_2|\|_{L^2}^2+B_2\|\theta_{\epsilon}^{\frac{1}{2}}|\varphi_2|\|_{L^2}^2,\\
    |J_3|&\ge \left(\gamma -\frac{\gamma p\delta_3}{2}-\frac{|\mathbf{c}|\eta_3}{2}\right)\|\theta_{\epsilon}^{\frac{1}{2}}|\nabla \varphi_2|\|_{L^2}^2
    +\left(\omega -\frac{\gamma p}{2\delta_3}-\frac{|\mathbf{c}|}{2\eta_3}\right)\|\theta_{\epsilon}^{\frac{1}{2}}|\varphi_2|\|_{L^2}^2
    =:A_3\|\theta_{\epsilon}^{\frac{1}{2}}|\nabla \varphi_3|\|_{L^2}^2+B_3\|\theta_{\epsilon}^{\frac{1}{2}}|\varphi_3|\|_{L^2}^2.
\end{split}
\]
We choose $\delta_1,\delta_2,\delta_3>0$ as
\[
\delta_1=\sqrt{\frac{\alpha}{2\omega}},\ \ \delta_2=\sqrt{\frac{\beta}{\omega}},\ \ \delta_3=\sqrt{\frac{\gamma}{\omega}}. 
\]
Then, we can see that 
\[
\frac{\alpha -\frac{\alpha p\delta_1}{2}}{\frac{|\mathbf{c}|}{2}}>\frac{\frac{|\mathbf{c}|}{2}}{2\omega -\frac{\alpha p}{2\delta_1}},\ \ 
\frac{\beta -\frac{\beta p\delta_2}{2}}{\frac{|\mathbf{c}|}{2}}>\frac{\frac{|\mathbf{c}|}{2}}{\omega -\frac{\beta p}{2\delta_2}},\ \ 
\frac{\gamma -\frac{\gamma p\delta_3}{2}}{\frac{|\mathbf{c}|}{2}}>\frac{\frac{|\mathbf{c}|}{2}}{\omega -\frac{\gamma p}{2\delta_3}}.
\]
by (\ref{kappa_cond}). 
Therefore, we can choose $\eta_1,\eta_2,\eta_3>0$ such that 
$A_1,A_2,A_3>0$ and $B_1,B_2,B_3>0$. 

On the other hand, 
by (\ref{phi_3_vanish_infi}), 
there exists $R>0$ such that for any $x\in \R^d$ with $|x|>R$,
\begin{equation}\label{nabla_M_est}
|\nabla \cdot \varphi_3(x)|\le \delta M
\end{equation}
holds, 
where $M:=\min_{1\le j\le 3}\{A_j,B_j\}>0$ 
and $\delta >0$ is small constant. 
For $j=1,2$, we divide the integral 
in the definition of $J_j$ as
\[
\begin{split}
|J_j|&\le \int_{\R^d}|\nabla \cdot \varphi_3|\theta_{\epsilon}|\varphi_1||\varphi_2|dx
=\int_{|x|\ge R}|\nabla \cdot \varphi_3|\theta_{\epsilon}|\varphi_1||\varphi_2|dx
+\int_{|x|<R}|\nabla \cdot \varphi_3|\theta_{\epsilon}|\varphi_1||\varphi_2|dx
=:J_{\infty}+J_{0}. 
\end{split}
\]
For the first term, by (\ref{nabla_M_est}) and 
the Young inequality, we obtain
\[
J_{\infty}
\le \delta M\int_{|x|\ge R}\theta_{\epsilon}|\varphi_1||\varphi_2|dx
\le\frac{\delta M}{2}\left(\|\theta_{\epsilon}^{\frac{1}{2}}|\varphi_1|\|_{L^2}^2
+\|\theta_{\epsilon}^{\frac{1}{2}}|\varphi_2|\|_{L^2}^2\right).
\] 
For the second term, by the H\"older inequality and 
the Sobolev inequality, there exists $C>0$ such that
\[
J_{0}\le 
\|\theta_{\epsilon}\|_{L^{\infty}(|x|<R)}
\|\nabla \cdot \varphi_3\|_{L^2}\|\varphi_1\|_{L^4}\|\varphi_2\|_{L^4}
\le C\theta_{\epsilon}(R)\prod_{j=1}^3\|\varphi_j\|_{H^1}.
\]
For $j=3$, we note that by the integration by parts,
\[
\begin{split}
J_3&=(\varphi_1\cdot \overline{\varphi_2},\nabla \cdot (\theta_{\epsilon}\varphi_3))
=
(\varphi_1\cdot \overline{\varphi_2},\theta_{\epsilon}\nabla \cdot \varphi_3)
+(\varphi_1\cdot \overline{\varphi_2},(\nabla \theta_{\epsilon})\cdot \varphi_3)\\
&=(\theta_{\epsilon}\varphi_1,(\nabla \cdot \varphi_3)\varphi_2)
+(\varphi_1\cdot \overline{\varphi_2},(\nabla \theta_{\epsilon})\cdot \varphi_3)
\\
&=J_1+(\varphi_1\cdot \overline{\varphi_2},(\nabla \theta_{\epsilon})\cdot \varphi_3).
\end{split}
\]
For the second term, we have
\[
\begin{split}
|(\varphi_1\cdot \overline{\varphi_2},(\nabla \theta_{\epsilon})\cdot \varphi_3)|
&\le p\int_{\R^d}\theta_{\epsilon}|\varphi_1||\varphi_2||\varphi_3|dx
=p\left(\int_{|x|\ge R}\theta_{\epsilon}|\varphi_1||\varphi_2||\varphi_3|dx
+\int_{|x|<R}\theta_{\epsilon}|\varphi_1||\varphi_2||\varphi_3|dx\right)\\
&\le \frac{p\delta M}{2}\left(\|\theta_{\epsilon}^{\frac{1}{2}}|\varphi_1|\|_{L^2}^2
+\|\theta_{\epsilon}^{\frac{1}{2}}|\varphi_2|\|_{L^2}^2\right)
+pC\theta_{\epsilon}(R)\prod_{j=1}^3\|\varphi_j\|_{H^1}
\end{split}
\]
by (\ref{theta_deriva}) and the same argument as for $j=1,2$ because $\varphi_3$ satisfies
\[
\lim_{|x|\rightarrow \infty}\varphi_3(x)=0. 
\]
As the results, we obtain
\[
\begin{split}
&M\sum_{j=1}^3\left(\|\theta_{\epsilon}^{\frac{1}{2}}|\nabla \varphi_j|\|_{L^2}^2+\|\theta_{\epsilon}^{\frac{1}{2}}|\varphi_j|\|_{L^2}^2\right)
\le |J_1|+|J_2|+|J_3|\\
&\le \frac{3+p}{2}\delta M
\left(\|\theta_{\epsilon}^{\frac{1}{2}}|\varphi_1|\|_{L^2}^2
+\|\theta_{\epsilon}^{\frac{1}{2}}|\varphi_2|\|_{L^2}^2\right)
+(3+p)C\theta_{\epsilon}(R)\prod_{j=1}^3\|\varphi_j\|_{H^1}. 
\end{split}
\]
By choosing $\delta =\frac{1}{3+p}$, we have
\[
\sum_{j=1}^3\left(\|\theta_{\epsilon}^{\frac{1}{2}}|\nabla \varphi_j|\|_{L^2}^2+\|\theta_{\epsilon}^{\frac{1}{2}}|\varphi_j|\|_{L^2}^2\right)
\le 2(3+p)C\theta_{\epsilon}(R)\prod_{j=1}^3\|\varphi_j\|_{H^1}.
\]
Let $\epsilon \rightarrow +0$ and by using the monotone convergence theorem, 
the desired estimate (\ref{weaksol_decay}) is obtained. 
\end{proof}
\subsection*{Acknowedgements}
{\rm
\ \ 
The authors express deep gratitude to Professor Norihisa Ikoma for fruiteful comments. The first author is supported by Grant-in-Aid for Young Scientists Research (No.21K13825). The second author is supported by JST CREST Grant Number JPMJCR1913, Japan and Young Scientists Research (No.19K14581), Japan Society for the Promotion of Science.
}



\begin{thebibliography}{30}

\bibitem{A18}
A. H. Ardila, 
    {\it Orbital stability of standing waves for a system of nonlinear Schr\"odinger equa-
tions with three wave interaction}, Nonlinear Anal., {\bf 167} (2018), 1--20. 

\bibitem{BP22}
H. Bahouri and G. Perelman, 
    {\it Global well-posedness for the derivative nonlinear Schr\"odinger equation}, Invent. Math., {\bf 229} (2022), 639--688. 
    
\bibitem{BL83}
    H. Br\'ezis and E. H. Lieb, 
    {\it A relation between pointwise convergence of functions and convergence of functionals}, 
    Proc. Amer. Math. Soc. {\bf 88} (1983), 486--490.

\bibitem{CC04}
    M. Colin and T. Colin,
    {\it On a quasilinear Zakharov system describing laser-plasma interactions}, Diff. Int. Equas., {\bf 17} (2004), 297--330.

\bibitem{CCO06}
    M. Colin, T. Colin and M. Ohta
    {\it Stability of solitary waves for a system of nonlinear Schr\"odinger equations with three wave interaction}, Ann. Inst. H. Poincar\'e C Anal. Non Lin\'eaire., {\bf 26} (2009), 2211--2226.
    
\bibitem{CO06}
    M. Colin and M. Ohta,
    {\it Stability of solitary waves for derivative nonlinear Schr\"dinger equation}, Ann. Inst. H. Poincar\'e Anal. Non Lin\'eaire., {\bf 23} (2006), 753--764.

\bibitem{GW95}
    B. Guo and Y. Wu, 
    {\it Orbital stability of solitary waves for the nonlinear derivative Schr\"odinger equation}, J. Diff. Equas., {\bf 123} (1995), 35--55.

\bibitem{FHI17}
N. Fukaya, M. Hayashi and T. Inui, {\it A sufficient condition for global existence of solutions to a
generalized derivative nonlinear Schr\"odinger equation}, Anal. PDE {\bf 10} (2017), 1149--1167. 

\bibitem{FHI23}
N. Fukaya, M. Hayashi and T. Inui, {\it Traveling waves for a nonlinear Schr\"odinger system with quadratic interaction}, to appear in Mathematische Annalen. 

\bibitem{Hayashi22}
M. Hayashi, {\it Stability of algebraic solitons for nonlinear Schr\"odinger equations of derivative type: variational approach}, Ann. Henri Poinca\'e {\bf 23} (2022), 4249--4277. 

\bibitem{Hirayama14}
H. Hirayama, {\it Well-posedness and scattering for a system of quadratic derivative nonlinear Schr\"odinger equations with low regularity initial data}, Commun. Pure Appl. Anal., {\bf 13} (2014), no. 4, 1563--1591.

\bibitem{HK19}
H. Hirayama and S. Kinoshita, {\it Sharp bilinear estimates and its application to a system of quadratic derivative nonlinear Schr\"odinger equations}, Nonlinear Anal., {\bf 178} (2019), 205--226.

\bibitem{HHO20}
H. Hirayama, S. Kinoshita and M. Okamoto, {\it Well-posedness for a system of quadratic derivative nonlinear Schr\"odinger equations with radial initial data}, Ann. Henri Poincar\'e, {\bf 21} (2020), no. 8, 2611--2636.

\bibitem{HKO21}
H. Hirayama, S. Kinoshita and M. Okamoto, {\it Well-posedness for a system of quadratic derivative nonlinear Schr\"odinger equations in almost critical spaces}, J. Math. Anal. Appl., {\bf 499} (2021), no. 2, Paper No. 125028, 29 pp.

\bibitem{HKO22}
H. Hirayama, S. Kinoshita and M. Okamoto, {\it  remark on the well-posedness for a system of quadratic derivative nonlinear Schr\"odinger equations}, Comm. Pure. Appl. Anal., {\bf 21} (2022), no. 10, 3309--3334.

\bibitem{KW18}
S. Kwon and Y. Wu, {\it Orbital stability of solitary waves for derivative nonlinear Schr\"odinger equation}, J. Anal. Math., {\bf 135} (2018), no. 2, 473--486.

\bibitem{KO22}
K. Kurata and Y. Osada, 
    {\it Variational problems associated with a system of nonlinear Schr\"odinger equations with three wave interaction}, 
    Discrete Contin. Dyn. Syst. Ser. B, {\bf 27} (2022), 4239--4251.
    
\bibitem{Lieb83}
E. H. Lieb, 
    {\it On the lowest eigenvalue of the Laplacian for the intersection of two domains}, 
    Invent. Math. {\bf 74} (1983), 441--448.

\bibitem{WG12}
    Y. Z. Wang, W. Ge,
    {\it Well-posedness of initial value problem for fourth order nonlinear Schr\"odinger Equation}, Pure and Applied Mathematics Quarterly, {\bf 8} (2012), 1047--1073.

\bibitem{W15}
    Y. Wu,
    {\it Global well-posedness on the derivative nonlinear Schr\"odinger equation}, Anal. PDE, {\bf 8} (2015), 1101--1112.

\bibitem{YC06}
    J. Yang, X-J. Chen,
    {\it Linearization operators of solutions in the derivative nonlinear Schr\"odinger hierarchy}, Trends in Soliton Research, (2006), 15--27.

\end{thebibliography}

\end{document}